\documentclass[12pt]{article}

\usepackage{mathrsfs}
\usepackage{enumitem}
\usepackage{overpic}
\usepackage{rotating}
\usepackage[colorlinks, linkcolor=blue, citecolor=blue]{hyperref}
\usepackage{color}
\usepackage{graphicx,subfigure,amsmath,amssymb,amsfonts,bm,epsfig,epsf,url,dsfont}
\usepackage{amsthm}
\usepackage{tikz}
\usetikzlibrary{arrows.meta, bending, positioning,calc}
\usepackage{bbm}      
\usepackage{booktabs}
\usepackage{cases}
\usepackage{fullpage}
\usepackage[small,bf]{caption}
\usepackage[top=1in,bottom=1in,left=1in,right=1in]{geometry}
\usepackage{fancybox}
\usepackage{algorithm}
\usepackage{verbatim}
\usepackage{algorithmic}
\usepackage{mathtools}
\usepackage{scalerel,stackengine}

\definecolor{myTeal}{RGB}{0, 128, 128}
\definecolor{myCoral}{RGB}{255, 127, 80}

\newcommand{\Var}{\text{Var}}

\newcommand{\KL}{D_{\mathrm{KL}}}

\newtheorem{definition}{Definition}

\newtheorem{theorem}{Theorem}
\newtheorem{corollary}[theorem]{Corollary}

\newtheorem{lemma}{Lemma}

\newtheorem{rmk}{Remark}

\newenvironment{fminipage}%
  {\begin{Sbox}\begin{minipage}}%
  {\end{minipage}\end{Sbox}\fbox{\TheSbox}}

\newcommand{\calA}{{\cal A}}
\newcommand{\calB}{{\cal B}}

\newcommand{\calD}{{\cal D}}
\newcommand{\calE}{{\cal E}}
\newcommand{\calF}{{\cal F}}

\newcommand{\calH}{{\cal H}}

\newcommand{\calK}{{\cal K}}
\newcommand{\calL}{{\cal L}}
\newcommand{\calM}{{\cal M}}
\newcommand{\calN}{{\cal N}}

\newcommand{\calP}{{\cal P}}
\newcommand{\calQ}{{\cal Q}}

\newcommand{\calS}{{\cal S}}

\newcommand{\calU}{{\cal U}}

\newcommand{\be}{\begin{equation}}
\newcommand{\ee}{\end{equation}}
\newcommand{\beqna}{\begin{eqnarray}}
\newcommand{\eeqna}{\end{eqnarray}}

\DeclarePairedDelimiterX{\set}[1]{\{}{\}}{\setargs{#1}}
\DeclarePairedDelimiterX{\cond}[1]{[}{]}{\setargs{#1}}
\NewDocumentCommand{\setargs}{>{\SplitArgument{1}{;}}m}
{\setargsaux#1}
\NewDocumentCommand{\setargsaux}{mm}
{\IfNoValueTF{#2}{#1} {#1\,\delimsize|\,\mathopen{}#2}}

\newcommand{\indep}{\perp \!\!\! \perp}

\makeatletter

\newcommand{\Ind}{\mathbbm{1}}
\newcommand{\R}{\mathbb{R}} 

\def\P{{\mathbb P}}
\newcommand{\p}[1]{\left(#1\right)}
\newcommand{\pp}[1]{\left[#1\right]}
\newcommand{\ppp}[1]{\left\{#1\right\}}
\newcommand{\norm}[1]{\left\|#1\right\|}
\newcommand{\s}[1]{\mathsf{#1}}

\newcommand{\E}{\mathbb{E}}
\newcommand{\abs}[1]{\left\lvert #1 \right\rvert}

\usepackage{titlesec}

\titlespacing*{\paragraph}
{0pt}   
{1.25ex plus 1ex minus .2ex}  
{1em}   

\makeatletter
\def\thanks#1{\protected@xdef\@thanks{\@thanks
        \protect\footnotetext{#1}}}
\makeatother

\definecolor{morpink}{RGB}{214, 72, 145}

\title{Inhomogeneous Submatrix Detection}
\author{Mor~Oren-Loberman~~~~Dvir Jerbi~~~~Tamir Bendory~~~~Wasim Huleihel\thanks{All authors are with the Department of Electrical and Computer Engineering-Systems at Tel Aviv University, {T}el {A}viv 6997801, Israel (e-mails:  \texttt{orenmor@mail.tau.ac.il, dvirjerbi@mail.tau.ac.il, bendory@tauex.tau.ac.il, wasimh@tauex.tau.ac.il}). This work is supported by the ISRAEL SCIENCE FOUNDATION (grant No. 1734/21). T.B. is supported in part by BSF under Grant 2020159, in part by NSF-BSF under Grant 2024791, in part by ISF under Grant 1924/21, and in part by a grant from The Center for AI and Data Science at Tel Aviv University (TAD).}}

\begin{document}
\maketitle
\sloppy
\begin{abstract}
 In this paper, we study the problem of detecting multiple hidden submatrices in a large Gaussian random matrix when the planted signal is inhomogeneous across entries. Under the null hypothesis, the observed matrix has independent and identically distributed standard normal entries. Under the alternative, there exist several planted submatrices whose entries deviate from the background in one of two ways: in the mean-shift model, planted entries (templates) have nonzero and possibly varying means; in the variance-shift model, planted entries have inflated and possibly varying variances. We consider two placement regimes for the planted submatrices. In the first, the row and column index sets are arbitrary. Motivated by scientific applications, in the second regime the row and column indices are restricted to be consecutive. For both alternatives and both placement regimes, we analyze the statistical limits of detection by proving information-theoretic lower bounds and by designing algorithms that match these bounds up to logarithmic factors, for a wide family of templates. 
\end{abstract}

\allowdisplaybreaks
\section{Introduction}

This paper investigates the problem of detecting hidden submatrices embedded in a large Gaussian random matrix. Under the null hypothesis, the observed $n \times n$ matrix consists of independent and identically distributed standard normal entries. Under the alternative hypothesis, there exist $m$ disjoint submatrices of size $k \times k$ whose entries follow Gaussian distributions with possibly non-uniform means or variances. The objective is to construct a test, equivalently an algorithm, that reliably distinguishes between these two hypotheses.

We begin by considering two models for the placement of the planted submatrices. In the first model, the planted blocks are disjoint and their row and column supports may be arbitrary subsets of indices. The detection and recovery versions of this formulation correspond to the well-studied problems of \emph{submatrix detection} and \emph{submatrix localization}, which have attracted considerable attention in recent years; see, for example, \cite{shabalin2009finding,kolar2011minimax,balakrishnan2011statistical,butucea2013detection,arias2014community,hajek2015computational,montanari2015limitation,verzelen2015community,ma2015computational,XingNobel,Arias10,Bhamidi17,chen2016statistical,cai2015computational,brennan18a,brennan19,9917525,rotenberg2024planted,elimelech2025detecting,9917525} and the references therein.

In the canonical setting of a single planted block, one seeks to determine whether an $n \times n$ matrix sampled from a distribution $\mathcal{Q}$ contains a hidden $k \times k$ submatrix whose entries are drawn from a different distribution $\mathcal{P}$. When both $\mathcal{P}$ and $\mathcal{Q}$ are Gaussian, the statistical limits of detection, optimal testing procedures, and computational lower bounds have been characterized in detail in \cite{butucea2013detection,montanari2015limitation,shabalin2009finding,kolar2011minimax,balakrishnan2011statistical,ma2015computational,brennan19}. If instead $\mathcal{P}$ and $\mathcal{Q}$ are Bernoulli distributions, the problem reduces to the planted dense subgraph model, which has also been extensively investigated; see, for example, \cite{butucea2013detection,arias2014community,verzelen2015community,hajek2015computational,brennan18a}. A central insight emerging from this line of work, both in the Gaussian and Bernoulli settings, is the existence of a statistical computational gap: the minimal signal size $k$ required for information-theoretic detectability can be strictly smaller than the minimal size for which detection is achievable by polynomial-time algorithms. The recovery problem has likewise been studied extensively, across different distributional assumptions and in both single and multiple block regimes; see \cite{chen2016statistical,montanari2015finding,candogan2018finding,hajek2016achieving,hajek2016information,cai2015computational,brennan18a}. These works characterize when the support of the planted submatrix can be accurately identified and analyze both information-theoretic and computational thresholds.

The arbitrary placement model described above is particularly natural in applications such as biclustering. Biclustering methods aim to uncover structured submatrices, often interpreted as hidden clusters, within large data tables of samples versus variables. Such techniques arise in a variety of domains, including community detection in social networks and the analysis of gene expression data from microarrays. Nevertheless, in certain scientific and engineering contexts, assuming arbitrary index sets may be unrealistic. Motivated by such considerations, we also study a structured placement model in which each planted block is supported on consecutive rows and consecutive columns. A prominent example arises in single-particle cryo-electron microscopy, a leading experimental technique for determining the three-dimensional structures of macromolecules, such as proteins \cite{bai2015cryo,lyumkis2019challenges}. Early stages of the cryo-electron microscopy computational pipeline require locating numerous particle images, which are noisy two-dimensional projections of randomly oriented molecular copies, within a large noisy micrograph \cite{singer2018mathematics,bendory2020single}. This step, known as particle picking, is algorithmically challenging. While various heuristic and learning-based approaches have been proposed, including \cite{wang2016deeppicker,heimowitz2018apple,bepler2019positive,eldar2020klt,eldar2024object}, a systematic investigation of the fundamental statistical and computational limits of this detection problem remains largely unexplored.

Our work is most closely related to \cite{dadon2024detection}, which analyzes detection and recovery of multiple homogeneous Gaussian submatrices under both arbitrary and consecutive placements and establishes sharp information-theoretic and computational thresholds. In that setting, each planted block is homogeneous, meaning that its entries share a common distribution across the entire block. In contrast, the present work develops a more general framework that allows structured heterogeneity within each planted submatrix. We introduce a finite-template model in which every block is assigned one template from a fixed finite collection, and the distribution of each entry depends on its relative coordinate within the block and on the chosen template. The classical homogeneous model is recovered as a special case when all templates coincide and are constant across coordinates. This generalization is motivated by the observation that, in many realistic scenarios, signals are not spatially uniform. Instead, they may exhibit gradients, anisotropies, or other structured patterns that cannot be captured by a single mean or variance parameter. In particular, allowing coordinate-dependent structure fundamentally changes both the statistical and analytical landscape of the problem. The detectability threshold is no longer governed by a single scalar shift, and the interaction between heterogeneous templates and random block overlaps must be handled explicitly. Addressing these challenges requires new probabilistic tools and a substantially more delicate second-moment Bayes risk analysis that captures how coordinate-dependent signals interact through random block overlaps; phenomena that are absent in the homogeneous setting. We now formalize this framework in the Gaussian setting and summarize our main results.

We focus on two Gaussian variants of this framework: mean-shift templates and variance-shift templates. In the mean-shift model, each template specifies a coordinate-wise mean profile over the $k \times k$ block, while the noise variance remains fixed. In the variance-shift model, the mean remains zero and the template assigns coordinate-wise variances. For both variants, we analyze arbitrary, non-consecutive placements as well as consecutive placements in which each block occupies a Cartesian product of row and column intervals of length $k$. These two placement regimes differ in their combinatorial complexity and in the distribution of overlaps between candidate blocks, leading to distinct detection thresholds. In the consecutive setting, we additionally study a circular variant in which indices are taken modulo $n$. This restores translation invariance and facilitates the lower-bound analysis without altering the essential statistical behavior.

Our main results establish information-theoretic lower bounds for detection in the finite-template model via a second-moment analysis of the likelihood ratio under the null hypothesis. The resulting threshold is characterized by a scalar quantity determined by the entrywise $\chi^2$-divergences associated with the templates and by the overlap distribution induced by the placement family. Under a mild regularity condition that excludes highly concentrated templates, this quantity reduces to a natural signal energy parameter. On the algorithmic side, we study several simple and computationally efficient test statistics, including global linear and quadratic statistics as well as template-matched scan procedures. These tests succeed in complementary parameter regimes. Under the same regularity condition, we identify parameter ranges in which these procedures match the information-theoretic detection boundary, up to logarithmic factors.

\paragraph{Notation.} For any positive integer $k$, let $[k]\triangleq\{1,2,\dots,k\}$. We write $\abs{\calU}$ for the cardinality of a finite set $\calU$, and $\Ind\ppp{\mathscr{E}}$ for the indicator of an event $\mathscr{E}$. By convention, $\frac{1}{|\calU|}\sum_{u\in \calU} g(u)=1$ when $\calU=\emptyset$. For a matrix $M\in\mathbb{R}^{k\times k}$, we use the standard norms $\norm{M}_1 \triangleq\sum_{i,j\in[k]}\abs{M_{ij}}$, $\norm{M}_F^2 \triangleq\sum_{i,j\in[k]} M_{ij}^2$, and $\norm{M}_{\infty} \triangleq\max_{i,j\in[k]}\abs{M_{ij}}$.

We write $\mathcal{N}(\mu,\sigma^2)$ for a univariate Gaussian distribution with mean $\mu$ and variance $\sigma^2$. We write $\s{X}\indep \s{Y}$ when random variables $\s{X}$ and $\s{Y}$ are independent.
For probability measures $\mathbb{P}\ll \mathbb{Q}$ on a common measurable space, we write $\chi^2(\mathbb{P}\Vert\mathbb{Q})\triangleq\int(\mathrm{d}\mathbb{P}/{\mathrm{d}\mathbb{Q}}-1)^2\mathrm{d}\mathbb{Q}$ for the chi-square divergence. We also write $d_{\s{TV}}(\mathbb{P},\mathbb{Q})\triangleq\frac{1}{2}\int\abs{\mathrm{d}\mathbb{P} -\mathrm{d}\mathbb{Q}}$ and $d_{\s{KL}}(\mathbb{P}\Vert \mathbb{Q})
\triangleq
\int\mathrm{d}\mathbb{P}\log\frac{\mathrm{d}\mathbb{P}}{\mathrm{d}\mathbb{Q}}$ for total variation and Kullback--Leibler divergence, respectively. For a matrix $\Sigma\in\mathbb{R}^{k\times k}$ with entries satisfying $0\le \Sigma_{ij}<1$, we define the blockwise Kullback--Leibler divergence $\s{KL}(\Sigma)\triangleq\frac{1}{2} \sum_{i,j\in[k]} \p{\Sigma_{ij}-\log(1+\Sigma_{ij})}$.

We use standard asymptotic notation: for sequences $\{a_n\}$ and $\{b_n\}$, we write $a_n=O(b_n)$, $a_n=\Omega(b_n)$, $a_n=\Theta(b_n)$, $a_n=o(b_n)$, and $a_n=\omega(b_n)$ with their usual meanings as $n\to\infty$. All asymptotic statements are taken in the limit $n\to\infty$, with $m=m(n)$ and $k=k(n)$ possibly depending on $n$ unless stated otherwise. For $a,b\in\mathbb{R}$, we write $a\vee b\triangleq\max\{a,b\}$ and $a\wedge b\triangleq\min\{a,b\}$. 
Throughout the paper, $C$ denotes a positive constant whose value may change from line to line and which is independent of $n$, $m$, $k$, and all signal parameters; any dependence on fixed model objects, such as the template family, will be stated explicitly when relevant.

\section{Problem Formulation} \label{sec:prob_form}
In this section, we introduce and formulate the statistical models and detection problems, analyzed in this paper. 

\paragraph{Blocks, placements, and indexing.}
Let $m,k,n\in\mathbb{N}$ satisfy $mk\le n$. A \emph{block} (or \emph{submatrix}) is a Cartesian product $\mathcal{B}=\mathcal{S}\times\mathcal{T}\subset[n]\times[n]$ where $\mathcal{S},\mathcal{T}\subset[n]$ and $|\mathcal{S}|=|\mathcal{T}|=k$. Note that the entries of $\mathcal{S}$ and $\mathcal{T}$ need not be consecutive integers, i.e., the block may be scattered rather than contiguous. Let
\begin{align}
    \calB_{k,n}
    \triangleq
    \ppp{\s{S}\times\s{T}:\s{S},\s{T}\subset[n],\ \abs{\s{S}}=\abs{\s{T}}=k}, \label{eq:block_set}
\end{align}
be the family of all $k\times k$ blocks, and define the family of unordered collections of
$m$ pairwise disjoint blocks by
\begin{align}
    \calK_{k,m,n}
    \triangleq
    \ppp{\s{K}\subset\calB_{k,n}:\abs{\s{K}}=m,\;
    \s{B}\cap\s{B}'=\emptyset\ \forall \s{B}\neq\s{B}'\in\s{K}}. \label{eq:non_con_B_K_def}
\end{align}
In the \emph{consecutive} placement model, let
$\mathcal{C}_{k,n}^{\s{con}} \triangleq \ppp{\{i+1,\dots,i+k\}: i=0,\dots,n-k}$ and set
\begin{align}
    \calB_{k,n}^{\s{con}}
    &\triangleq
    \ppp{\s{S}\times\s{T}:\s{S},\s{T}\in\mathcal{C}_{k,n}^{\s{con}}},\label{eq:block_set_con}
    \\
    \calK_{k,m,n}^{\s{con}} 
    &\triangleq
    \ppp{\s{K}\subset\calB_{k,n}^{\s{con}}:\abs{\s{K}}=m,\ 
    \s{B}\cap\s{B}'=\emptyset\ \forall \s{B}\neq\s{B}'\in\s{K}}. 
\end{align}
In addition to the standard consecutive placement model, we also consider a \emph{circular consecutive} variant, in which indices are taken modulo $n$. Define $\mathcal{C}_{k,n}^{\circ} \triangleq\big\{\ppp{i+1,\dots,i+k} \bmod n:\ i=0,\dots,n-1\big\}$, and set
\begin{align}
    \calB_{k,n}^{\circ}     &\triangleq    \ppp{\s{S}\times\s{T}:\s{S},\s{T}\in\mathcal{C}_{k,n}^{\circ}}, \label{eq:block_set_circ}\\
    \calK_{k,m,n}^{\circ} &\triangleq \ppp{\s{K}\subset\calB_{k,n}^{\circ}:\abs{\s{K}}=m,\ \s{B}\cap\s{B}'=\emptyset\ \forall \s{B}\neq\s{B}'\in\s{K}}.   
\end{align}
The circular consecutive model restores translation symmetry across block locations. Note that $|\mathcal{B}_{k,n}^{\s{con}}|=(n-k+1)^2$, whereas $|\mathcal{B}_{k,n}^{\circ}|=n^2$. When both models are considered, results for the circular consecutive placement are stated explicitly, with extensions to the standard consecutive placement specified when applicable. Figure~\ref{fig:ft_model_schematic} illustrates the placement regimes introduces above.

Finally, let $\mathcal{B}=\mathcal{S}\times\mathcal{T}$ be a block with $\mathcal{S}=\{s_1<\cdots<s_k\}$ and $\mathcal{T}=\{t_1<\cdots<t_k\}$. We define the induced coordinate map
\begin{align}
    \varphi_{\s{B}}:\mathcal{S}\times\mathcal{T}\to[k]\times[k],
    \qquad
    \varphi_{\s{B}}(s_u,t_v)=(u,v).
    \label{eq:induced_coord_map}
\end{align}
That is, $\varphi_{\mathcal{B}}$ assigns to each entry of $\mathcal{B}$ its relative row and column indices within the block. We emphasize that the map $\varphi_{\mathcal{B}}$ is deterministic and not part of the generative model; it is an indexing convention induced by the ordering of $\mathcal{S}$ and $\mathcal{T}$. Its role is to align the entries of a block located at arbitrary coordinates in $[n]\times[n]$ with a canonical $k\times k$ template indexed by $[k]\times[k]$. We remark below why this alignment is essential in the inhomogeneous submatrix model. Figure~\ref{fig:ft_mapping_illustration} illustrates the alignment induced by $\varphi_{\s{B}}$.

\begin{figure}[t]
    \centering
    \begin{tikzpicture}[scale=0.35]

        \begin{scope}
            \draw[step=1, gray!20, thin] (0,0) grid (16,16);
            \draw[line width=0.8pt, gray!80] (0,0) rectangle (16,16);

            \foreach \x in {0, 4, 5, 12} {
                \foreach \y in {1, 2, 3, 11} {
                    \fill[myTeal, opacity=0.8, rounded corners=0.5pt] (\x+0.1, \y+0.1) rectangle (\x+0.9, \y+0.9);
                }
            }

            \foreach \x in {2, 7, 10, 11} {
                \foreach \y in {5, 7, 8, 14} {
                    \fill[myCoral, opacity=0.8, rounded corners=0.5pt] (\x+0.1, \y+0.1) rectangle (\x+0.9, \y+0.9);
                }
            }

            \node[below=0.2cm] at (8,0) { $\mathcal{K}_{k,m,n}$};
        \end{scope}

        \begin{scope}[xshift=22cm]
            \draw[step=1, gray!20, thin] (0,0) grid (16,16);
            \draw[line width=0.8pt, gray!80] (0,0) rectangle (16,16);

            \fill[myTeal, opacity=0.8, rounded corners=1pt] (1.1, 1.1) rectangle (4.9, 4.9);

            \fill[myCoral, opacity=0.8, rounded corners=1pt] (10.1, 10.1) rectangle (13.9, 13.9);

            \node[below=0.2cm] at (8,0) { $\mathcal{K}_{k,m,n}^{\s{con}}$};
        \end{scope}

    \end{tikzpicture}
    \caption{Schematic illustration of the placement families (shown for $n=16$, $k=4$). In the non-consecutive model, arbitrary row and column subsets are selected, yielding blocks of the form $\s{S}\times\s{T}$. In the consecutive model, the row and column sets are intervals of length $k$.}
    \label{fig:ft_model_schematic}
\end{figure}
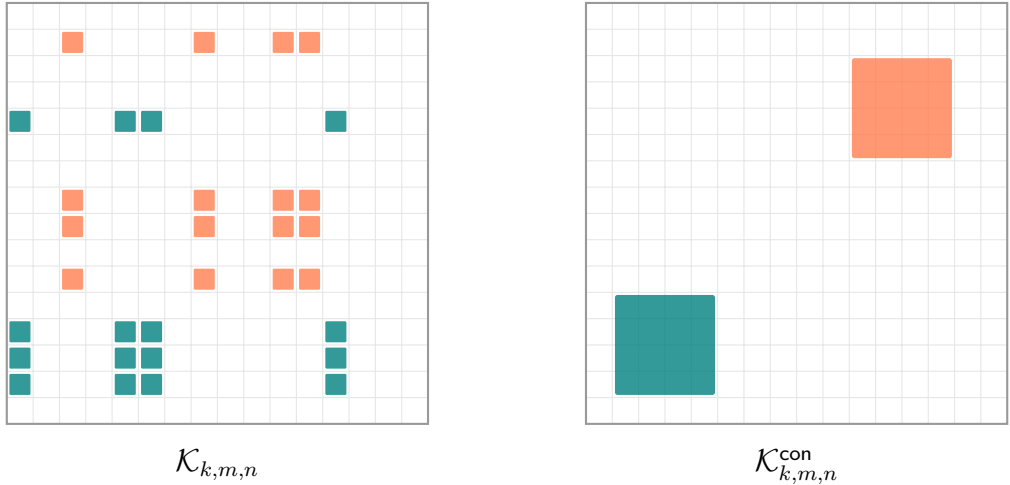

\begin{figure}[t]
    \centering
    \begin{tikzpicture}[scale=0.38]

        \begin{scope}
            \draw[step=1, gray!20, thin] (0,0) grid (16,16);
            \draw[line width=0.8pt, gray!80] (0,0) rectangle (16,16);

            \foreach \x in {2, 3, 4, 5} {
                \foreach \y in {8, 9, 10, 11} {
                    \pgfmathsetmacro{\isTarget}{ (\x == 4 && \y == 10) ? 1 : 0 }
                    \ifnum\isTarget=1
                        \fill[myTeal, opacity=1, rounded corners=0.5pt] (\x+0.1, \y+0.1) rectangle (\x+0.9, \y+0.9);
                        \node (tealSource) at (\x+0.5, \y+0.5) {};
                    \else
                        \fill[myTeal, opacity=0.2, rounded corners=0.5pt] (\x+0.1, \y+0.1) rectangle (\x+0.9, \y+0.9);
                    \fi
                }
            }

            \foreach \x [count=\v] in {7, 8, 12, 15} {
                \node[below, myCoral, font=\scriptsize] at (\x+0.5, 0) {$t_{\v}$};
                \foreach \y [count=\u] in {14, 3, 1, 0} {
                    \pgfmathsetmacro{\isTarget}{ (\u == 2 && \v == 3) ? 1 : 0 }
                    \ifnum\isTarget=1
                        \fill[myCoral, opacity=1, rounded corners=0.5pt] (\x+0.1, \y+0.1) rectangle (\x+0.9, \y+0.9);
                        \node (coralSource) at (\x+0.5, \y+0.5) {};
                    \else
                        \fill[myCoral, opacity=0.2, rounded corners=0.5pt] (\x+0.1, \y+0.1) rectangle (\x+0.9, \y+0.9);
                    \fi
                }
            }
            \foreach \y [count=\u] in {14, 3, 1, 0} {
                \node[left, myCoral, font=\scriptsize] at (0, \y+0.5) {$s_{\u}$};
            }
        \end{scope}

        
        \begin{scope}[xshift=25cm, yshift=10cm]
            \draw[step=1, gray!20, thin] (0,0) grid (4,4);
            \draw[line width=0.8pt, gray!80] (0,0) rectangle (4,4);
            \foreach \tx in {0,...,3} \foreach \ty in {0,...,3} \fill[myTeal, opacity=0.1] (\tx+0.1, \ty+0.1) rectangle (\tx+0.9, \ty+0.9);
            \fill[myTeal, opacity=1, rounded corners=0.5pt] (2.1, 2.1) rectangle (2.9, 2.9);
            \node (tealDest) at (2.5, 2.5) {};
            \node[right=0.2cm, myTeal] at (4,2) {$\mathbf{M_1}$};
        \end{scope}

        \begin{scope}[xshift=25cm, yshift=1cm]
            \draw[step=1, gray!20, thin] (0,0) grid (4,4);
            \draw[line width=0.8pt, gray!80] (0,0) rectangle (4,4);
            \foreach \tx in {0,...,3} \foreach \ty in {0,...,3} \fill[myCoral, opacity=0.1] (\tx+0.1, \ty+0.1) rectangle (\tx+0.9, \ty+0.9);
            \fill[myCoral, opacity=1, rounded corners=0.5pt] (2.1, 2.1) rectangle (2.9, 2.9);
            \node (coralDest) at (2.5, 2.5) {};
            \node[right=0.2cm, myCoral] at (4,2) {$\mathbf{M_2}$};
        \end{scope}

        \draw[->, >=Stealth, thick, dashed, myTeal, bend left=20] 
            (tealSource) to node[midway, xshift=1.7cm, yshift=0.33cm, font=\small, text=myTeal] {$\varphi_{B_1}(s_2, t_3) = (2, 3)$} (tealDest);
            
        \draw[->, >=Stealth, thick, dashed, myCoral, bend right=20] 
            (coralSource) to node[midway, xshift=0.2cm, yshift=-0.3cm, font=\small, text=myCoral] {$\varphi_{B_2}(s_2, t_3) = (2, 3)$} (coralDest);

    \end{tikzpicture}
    \caption{Illustration of the coordinate map $\varphi_{\s{B}}$ in \eqref{eq:induced_coord_map}. For a block $\s{B}=\s{S}\times\s{T}$, the map $\varphi_{\s{B}}(i,j)=(u,v)$ records the relative row and column indices of $(i,j)$ within $\s{B}$, thereby aligning entries of $\s{B}$ with template coordinates. The figure shows both a consecutive block $\s{B}_1$ and a non-consecutive block $\s{B}_2$ mapped to their respective templates.}
    \label{fig:ft_mapping_illustration}
\end{figure}
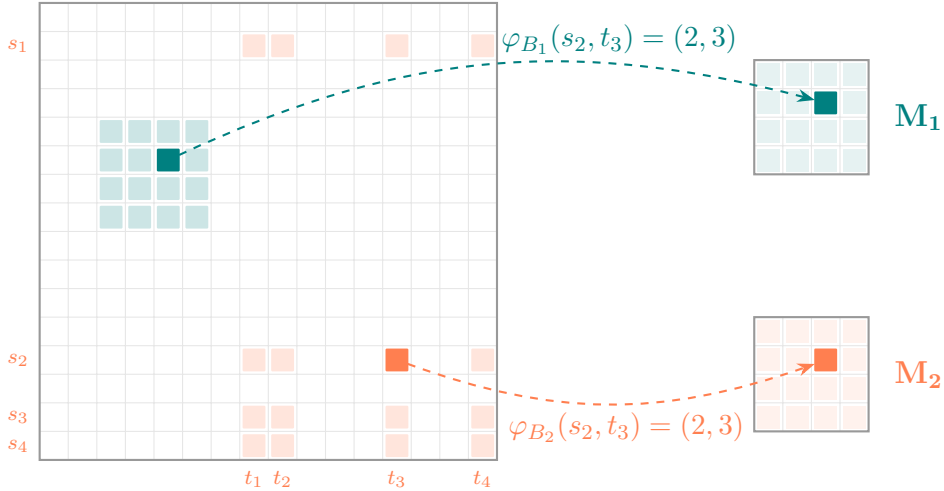

\paragraph{Finite-template submatrix detection model.}

Throughout, we focus on two Gaussian settings: submatrix detection under a mean shift and under a variance shift. For notational convenience, we write $\calQ \triangleq \calN(0,1)$ for the standard normal distribution, and $\calQ^{\otimes n\times n}$ denotes the corresponding i.i.d. product measure on $\mathbb{R}^{n\times n}$. 
Furthermore, we consider a finite family of templates indexed by $\ell\in[m]$. For each $\ell$, let
$\{\calP_{\ell,u}\}_{u\in[k]\times[k]}$ be a collection of probability distributions on $\R$.
The detection problem is to test
\begin{align}
    \calH_0:\ \s{X}\sim\calQ^{\otimes n\times n}
    \qquad \text{vs.} \qquad
    \calH_1:\ \s{X}\sim\calD(n,k,m,\{\calP_{\ell,u}\},\calQ),
\end{align}
where $\calD(n,k,m,\{\calP_{\ell,u}\},\calQ)$ is defined as follows. Under $\calH_1$, an unordered collection $\s{K}$ of $m$ pairwise disjoint blocks is drawn uniformly at random from a specified placement family: $\calK_{k,m,n}$ in the non-consecutive case, $\calK_{k,m,n}^{\s{con}}$ in the standard consecutive case, or $\calK_{k,m,n}^{\circ}$ in the circular consecutive case.
Conditional on $\s{K}$, a bijection $\beta:\s{K}\to[m]$ is drawn uniformly at random from all such bijections. Given $(\s{K},\beta)$, the entries of $\s{X}$ are independent and satisfy
\begin{align}
    \s{X}_{ij}\sim
    \begin{cases}
    \calP_{\beta(\s{B}),\,\varphi_{\s{B}}(i,j)}, & (i,j)\in\s{B}\text{ for some }\s{B}\in\s{K},\\
    \calQ, & \text{otherwise}.
    \end{cases}
\label{eq:generic_entries_law}
\end{align}

We study two Gaussian specializations of the finite-template model.
\begin{itemize}
    \item \textbf{Mean-shift model.} Let $\calM=\{M_\ell\}_{\ell=1}^m$ with $M_\ell\in\R^{k\times k}$.
    For $(u,v)\in[k]\times[k]$, define
    \begin{align}
        \calP_{\ell,(u,v)}=\calN\p{(M_\ell)_{uv},1}.
    \end{align}
    \item \textbf{Variance-shift model.} Let $\calS=\{\Sigma_\ell\}_{\ell=1}^m$ with $\Sigma_\ell\in\R_+^{k\times k}$
    satisfying $\max_{u,v}(\Sigma_\ell)_{uv}\le\vartheta_0$ for some fixed $\vartheta_0\in[0,1)$ that is independent of $n$.
    Define
    \begin{align}
        \calP_{\ell,(u,v)}=\calN\p{0,1+(\Sigma_\ell)_{uv}}.
    \end{align}
\end{itemize}
Throughout the paper, we denote by $\P_{\calH_0}$ the null distribution $\calQ^{\otimes n\times n}$ and by $\P_{\calH_1}$ the mixture alternative induced by the finite-template model under the placement family specified in each statement. When necessary, we write $\P_{\calH_1}^{\s{con}}$ and $\P_{\calH_1}^{\circ}$ to distinguish between the standard and circular consecutive placement models.

Finally, we would like to remark why the alignment in \eqref{eq:induced_coord_map} is essential in the inhomogeneous submatrix model. In contrast to the classical homogeneous setting, where all entries of the planted submatrix share the same distribution (e.g., Gaussians with the same mean as in \cite{ma2015computational,dadon2024detection}), here the distributions may vary across positions within the block. The signal is therefore specified in local block coordinates, say by a template matrix $M=(M_{u,v})_{u,v\in[k]}$. The map $\varphi_{\mathcal{B}}$ allows us to embed this fixed template into any candidate block $\mathcal{B}$ by setting, for $(i,j)\in\mathcal{B}$, $\mathbb{E}[\s{X}_{i,j}] = M_{\varphi_{\mathcal{B}}(i,j)}$. Thus, $\varphi_{\mathcal{B}}$ provides a canonical identification between global matrix coordinates and local block coordinates.

\paragraph{Goal.}
Given an observation $\s{X}$, a detection algorithm $\mathcal{A}_n$ outputs a decision in $\{0,1\}$, corresponding to the null and alternative hypotheses.
The risk of $\mathcal{A}_n$ is defined as the sum of its Type~I and Type~II error probabilities,
\begin{equation}
    \mathsf{R}(\mathcal{A}_n)
    \triangleq
    \P_{\calH_0}\p{\mathcal{A}_n(\s{X})=1}
    +
    \P_{\calH_1}\p{\mathcal{A}_n(\s{X})=0},
\end{equation}
where $\P_{\calH_0}$ denotes the distribution of $\s{X}$ under the null hypothesis and $\P_{\calH_1}$ denotes the marginal distribution of $\s{X}$ under the alternative. In particular, $\P_{\calH_1}$ is the mixture distribution obtained by averaging over the random choice of the planted block collection $\s{K}$ and, conditional on $\s{K}$, over the random labeling $\beta:\s{K}\to[m]$ specified by the model.
Equivalently,
\begin{align}
    \P_{\calH_1}(\cdot)
    =
    \E_{\s{K},\beta}\pp{\P_{\calH_1\vert\s{K},\beta}(\cdot)}.
\end{align}
We say that $\mathcal{A}_n$ solves the detection problem if $\mathsf{R}(\mathcal{A}_n)\to 0$ as $n\to\infty$. The procedures considered in this paper are either unrestricted (and potentially computationally expensive) or restricted to run in polynomial time.

\section{Main Results}
In this section, we present our main results on submatrix detection under the finite-template model introduced in Section~\ref{sec:prob_form}. We identify parameter regimes in which detection is possible and those in which it is information-theoretically impossible. We also investigate detection under polynomial-time computational constraints.

\subsection{Upper bounds}

\subsubsection{Mean-shift model} We begin by establishing achievable detection guarantees for the finite-template mean-shift model. To this end, we introduce the proposed detection algorithms and then state their performance guarantees. 

\paragraph{Global test.} We first consider a global test based on aggregating all matrix entries. Define
\begin{align}
    \s{T}_{\s{sum}}(\s{X})
    \triangleq
    \s{sign}(\mu_{\s{det}})\sum_{i,j\in[n]} \s{X}_{ij},
    \qquad
    \mu_{\s{det}}
    \triangleq
    \sum_{\ell=1}^m \sum_{u,v\in[k]} (M_\ell)_{uv},
    \label{eq:Sum_stat}
\end{align}
where $\mu_{\s{det}}$ is the total planted mean mass contributed by all templates. Note that under the model, each template appears exactly once among the planted blocks by construction, so the total expected signal contribution equals $\mu_{\s{det}}$, independently of the block locations. 
Throughout, we assume $\mu_{\s{det}}\neq 0$. When $\mu_{\s{det}}=0$, the global sum statistic carries no signal, and we rely exclusively on scan-based procedures. Then, we define the \textit{global sum test} as
\begin{align}
    \calA_{\s{sum}}(\s{X})
    \triangleq
    \Ind\ppp{\s{T}_{\s{sum}}(\s{X})\ge \tau_{\s{sum}}}.
    \label{eq:sum_algo}
\end{align}
where $\tau_{\s{sum}}\triangleq\frac{|\mu_{\s{det}}|}{2}$. We note that $\s{T}_{\s{sum}}$ can be computed in linear time; thus, the sum test is computationally efficient.
\paragraph{Scan test.}
Let $M\in\R^{k\times k}$ be a fixed template matrix. For a given family of blocks $\calB$ (either $\calB_{k,n}$, $\calB_{k,n}^{\s{con}}$ or $\calB_{k,n}^{\circ}$), define the scan statistic
\begin{align}
    \s{T}^{\mu}_{\s{scan}}(\s{X};\calB,M)
    \triangleq
    \max_{\s{B}\in\calB}
    \sum_{(i,j)\in\s{B}}
    M_{\varphi_{\s{B}}(i,j)}\,\s{X}_{ij},
    \label{eq:Scan_stat}
\end{align}
where $\varphi_{\s{B}}$ is the induced coordinate map defined in \eqref{eq:induced_coord_map}. Note that \eqref{eq:Scan_stat} ``scans" w.r.t. a single template $M$ only. Intuitively, under the alternative hypothesis, each planted block is associated with a distinct template from the finite family $\calM=\{M_\ell\}_{\ell=1}^m$. For a fixed block location $\s{B}$, the expected value of the scan statistic in \eqref{eq:Scan_stat} equals $\|M_\ell\|_F^2$, when the planted template is $M_\ell$. Consequently, among the templates in $\calM$, the one with the largest Frobenius norm yields the largest expected shift for this statistic. Since detection requires only that at least one planted block produces a statistically significant excursion, it suffices to scan with the template achieving the maximal Frobenius norm. We therefore define
\begin{align}
    \ell_{\max}\in\arg\max_{\ell\in[m]} \norm{M_\ell}_F^2,
    \quad
    M_{\max}\triangleq M_{\ell_{\max}}.\label{eq:M_max}
\end{align}
Accordingly, the \textit{template-aware scan test} over a family of blocks $\calB$ is defined as
\begin{align}
    \calA^{\mu}_{\s{scan, max}}(\s{X};\calB)
    \triangleq
    \Ind\ppp{
        \s{T}^{\mu}_{\s{scan}}(\s{X};\calB,M_{\max})
        \ge \tau^{\mu}_{\s{scan}}(\calB)},
    \label{eq:scan_algo}
\end{align}
where
\begin{align}
    \tau^{\mu}_{\s{scan}}(\calB)
    \triangleq
    \begin{cases}
    \sqrt{(4+\delta)\,\norm{M_{\max}}_F^2\,k\log\dfrac{e n}{k}},
    & \calB=\calB_{k,n},\\[0.8em]
    \sqrt{(4+\delta)\,\norm{M_{\max}}_F^2\,\log n},
    & \calB\in\{\calB_{k,n}^{\s{con}},\,\calB_{k,n}^{\circ}\},
\end{cases}
    \label{eq:mean_thresholds}
\end{align}
for a fixed constant $\delta>0$. Note that the global sum test and the scan test for the standard and circular consecutive placement families are computationally efficient (running in polynomial time), whereas the scan test for the non-consecutive placement family is computationally expensive (having exponential time complexity). More related discussions are provided in Remark~\ref{rmk:consecutive_remark}. 

We are now in a position to state our main results.

\begin{theorem}[Mean-shift upper bounds]
\label{thrm:Detect_upper}
Consider the finite-template mean-shift model introduced in Section~\ref{sec:prob_form}. Let $M_{\max}$ be defined as in \eqref{eq:M_max}. The following statements hold.

\begin{enumerate}
    \item \label{item:mean_upper_sum} If
    \begin{align}
        |\mu_{\s{det}}| = \omega(n), \label{eq:sum_test_cond}
    \end{align}
    then the global sum test $\calA_{\s{sum}}(\s{X})$ in \eqref{eq:sum_algo} satisfies $\s{R}(\calA_{\s{sum}})=o(1)$ under the non-consecutive placement regime and under both consecutive placement regimes.

    \item \label{item:Upper_mean_scan_non_con}
    If
    \begin{align}
        \|M_{\max}\|_F^2
        =
        \omega\p{k \log{\frac{n}{k}}},
        \label{eq:IMSD_upper_bound}
    \end{align}
    then the template-aware scan test $\calA^{\mu}_{\s{scan, max}}(\s{X};\calB_{k,n})$ in     \eqref{eq:scan_algo} satisfies $\s{R}\p{\calA^{\mu}_{\s{scan, max}}}=o(1)$.

    \item \label{item:Upper_mean_scan_con}
    If
    \begin{align}
        \|M_{\max}\|_F^2
        =
        \omega(\log n),
    \end{align}
    then the template-aware scan test $\calA^{\mu}_{\s{scan, max}}(\s{X};\calB)$ in
    \eqref{eq:scan_algo} satisfies $\s{R}\p{\calA^{\mu}_{\s{scan, max}}}=o(1)$, under the standard consecutive placement family $\calB_{k,n}^{\s{con}}$; the same bound holds under the circular consecutive placement family $\calB_{k,n}^{\circ}$.
\end{enumerate}
\end{theorem}

\subsubsection{Variance-shift model}

We now move forward to the variance-shift detection model introduced in Section~\ref{sec:prob_form}. As in the mean-shift case, we consider both global and scan-based procedures. The resulting scan statistic coincides with the blockwise log-likelihood ratio under the Gaussian variance alternatives.   

\paragraph{Global test.}
Define the centered quadratic statistic
\begin{align}
    \s{T}_{\s{quad}}(\s{X})
    \triangleq
    \sum_{i,j\in[n]} \p{\s{X}_{ij}^2 - 1}.
    \label{eq:Sum_stat_var}
\end{align}
We note that under the null hypothesis, $\E_{\calH_0}[\s{T}_{\s{quad}}(\s{X})]=0$. Under the alternative, the mean increases by the total variance mass contributed by the planted blocks, namely, 
\begin{align}
    \nu_{\s{det}}
    \triangleq
    \sum_{\ell=1}^m \sum_{u,v\in[k]} (\Sigma_\ell)_{uv}.
    \label{eq:nu_var}
\end{align}
Accordingly, we define the \textit{global quadratic test} as
\begin{align}
    \calA_{\s{quad}}(\s{X})
    &\triangleq
    \Ind\{\s{T}_{\s{quad}}(\s{X}) \ge \tau_{\s{quad}}\},
    \label{eq:sum_algo_var}
\end{align}
where $\tau_{\s{quad}} \triangleq \frac{\nu_{\s{det}}}{2}$. 

\paragraph{Scan test.}
Let $\Sigma\in\R_+^{k\times k}$ be a fixed variance template satisfying
$\max_{u,v}\Sigma_{uv}\le \vartheta_0$ for a constant $\vartheta_0\in[0,1)$ independent of $n$. For a given family of blocks $\calB$ (either $\calB_{k,n}$, $\calB_{k,n}^{\s{con}}$ or $\calB_{k,n}^{\circ}$), define the scan statistic
\begin{align}
    \s{T}^{\sigma}_{\s{scan}}(\s{X};\calB,\Sigma)
    \triangleq 
    \max_{\s{B}\in\calB}
    \sum_{(i,j)\in\s{B}}
    \frac{1}{2}\p{
        \frac{\Sigma_{\varphi_{\s{B}}(i,j)}}{1+\Sigma_{\varphi_{\s{B}}(i,j)}}
        \s{X}_{ij}^2
        - \log\p{1+\Sigma_{\varphi_{\s{B}}(i,j)}}},
    \label{eq:Scan_stat_HVar}
\end{align}
where $\varphi_{\s{B}}$ is the induced coordinate map defined in \eqref{eq:induced_coord_map}. This statistic coincides with the blockwise log-likelihood ratio for a Gaussian variance-shift template. For such a template $\Sigma$, we further define its associated blockwise Kullback--Leibler divergence
\begin{align}
    \s{KL}(\Sigma)
    \triangleq
    \frac{1}{2} \sum_{u,v\in[k]}
    \pp{
        (\Sigma)_{uv}
        -
        \log\p{1+(\Sigma)_{uv}}}.
    \label{eq:KL_definition}
\end{align} 
Accordingly, we define the corresponding \textit{finite-template scan test} as
\begin{align}
    \calA^{\sigma}_{\s{scan}}(\s{X};\calB)
    &\triangleq
    \Ind\ppp{\max_{\ell \in [m]} \s{T}^{\sigma}_{\s{scan}}(\s{X};\calB,\Sigma_{\ell})
    \ge \tau_{\s{scan}}^{\sigma}(\calB)},
    \label{eq:scan_algo_var}
\end{align}
where
\begin{align}
    \tau_{\s{scan}}^{\sigma}(\calB)
    \triangleq
    \begin{cases}
        (1+\delta)\p{2k\log\frac{e n}{k} + \log m},
        & \calB=\calB_{k,n},\\[0.4em]
        (1+\delta)\p{2\log n + \log m},
        & \calB \in\ppp{\calB_{k,n}^{\s{con}},\calB_{k,n}^{\circ}}.
    \end{cases}
    \label{eq:variance_thresholds}
\end{align}
for a fixed constant $\delta>0$. We have the following result.

\begin{theorem}[Variance-shift upper bounds]
\label{thrm:Detect_upper_variance}
Consider the finite-template variance-shift model introduced in Section~\ref{sec:prob_form}. 
Then the following statements hold.
\begin{enumerate}
    \item If
    \begin{align}
        \nu_{\s{det}} = \omega(n), \label{eq:quad_cond}
    \end{align}
    then the global quadratic test $\calA_{\s{quad}}(\s{X})$ in
    \eqref{eq:sum_algo_var} satisfies $\s{R}(\calA_{\s{quad}})=o(1)$ under the non-consecutive placement regime and under both consecutive placement regimes.
    \item \label{item:variance_scan_non_con}
    If
    \begin{align}
        \max_{\ell} \s{KL}(\Sigma_{\ell})
        =
        \omega\p{k\log\frac{n}{k} + \log m},
        \label{eq:variance_upper_cond}
    \end{align}
    then the finite-template scan test $\calA^{\sigma}_{\s{scan}}(\s{X};\calB_{n,k})$ in \eqref{eq:scan_algo_var} satisfies $\s{R}\p{\calA^{\sigma}_{\s{scan}}}=o(1)$.
    
    \item  \label{item:variance_scan_con}
    If
    \begin{align}
        \max_{\ell} \s{KL}(\Sigma_{\ell}) = \omega(\log n +  \log m),
        \label{eq:con_variance_upper_cond}
    \end{align}
    then the finite-template scan test $\calA^{\sigma}_{\s{scan}}(\s{X};\calB)$ in  \eqref{eq:scan_algo_var} satisfies $\s{R}\p{\calA^{\sigma}_{\s{scan}}}=o(1)$, under the standard consecutive placement family $\calB_{k,n}^{\s{con}}$; the same bound holds under the circular consecutive placement family $\calB_{k,n}^{\circ}$.
\end{enumerate}
\end{theorem}


\begin{rmk}\label{rmk:template_asymmetry}
In the mean-shift setting, the scan statistic is linear in the data, and its expectation under the alternative equals $\|M_\ell\|_F^2$ when the planted template is $M_\ell$. Scanning with a template attaining $\max_{\ell}\|M_\ell\|_F$ therefore maximizes the expected value of this statistic among the planted blocks. This formulation isolates the dependence of the detection threshold on the block family $\calB$. Alternatively, one may scan jointly over block locations and templates, defining
\begin{align}
    \max_{\ell\in[m]}\max_{\s{B}\in\calB} \sum_{(i,j)\in \s{B}} (M_\ell)_{\varphi_{\s{B}}(i,j)} \s{X}_{ij}. \label{eq:finite-template_mean_stat}
\end{align}
By a union bound over $\ell\in[m]$, this procedure achieves vanishing risk whenever
\begin{align}
    \|M_{\max}\|_F^2 = \omega\p{\log|\calB|+\log m},
\end{align}
that is, under the same scaling as in Theorem~\ref{thrm:Detect_upper} with the scan threshold increased by an additive $\log m$ term. 

In contrast, in the variance-shift setting the scan statistic coincides with a blockwise log-likelihood ratio, and the relevant separation is quantified by the Kullback–Leibler divergence. In this case, scanning jointly over the finite template family aligns with the likelihood-ratio structure and is adopted in the stated upper bounds.
\end{rmk}

\begin{rmk} \label{rmk:consecutive_remark}
For the scan-based procedures in Theorems~\ref{thrm:Detect_upper} and~\ref{thrm:Detect_upper_variance}, the detection thresholds are governed by the cardinality of the collection $\calB$ of candidate block locations. In the variance-shift setting, where the scan is taken jointly over block locations and templates, the thresholds additionally depend on the size $m$ of the template family.
In the non-consecutive setting, $\calB_{k,n}$ defined in~\eqref{eq:block_set} satisfies $|\calB_{k,n}|=\binom{n}{k}^2\le (en/k)^{2k}$. Substituting this bound yields thresholds of order $k\log(en/k)$ for the mean-shift scan and $k\log(en/k)+\log m$ for the variance-shift scan. In the consecutive setting, both the standard placement family $\calB_{k,n}^{\s{con}}$ and the circular placement family $\calB_{k,n}^{\circ}$ satisfy $|\calB|=\Theta(n^2)$, yielding logarithmic thresholds of order $\log n$ for the mean-shift scan and $\log n+\log m$ for the variance-shift scan.

From a computational perspective, scanning over $\calB_{k,n}$ requires enumerating $\binom{n}{k}^2$ block locations and is computationally infeasible in general. In contrast, scans over $\calB_{k,n}^{\s{con}}$ and $\calB_{k,n}^{\circ}$ admit polynomial-time implementations via sliding-window or circular convolution techniques. In the variance-shift setting, the finite-template scan incurs an additional factor $m$ in the running time due to the enumeration over templates.
\end{rmk}


\subsection{Lower bounds} \label{sec:results_lowerbound}
We now present our information-theoretic lower bounds. Recall that the minimal (or, optimal) achievable risk for testing $\P_{\calH_0}$ versus $\P_{\calH_1}$ satisfies (see, e.g.,~\cite{tsybakov2004introduction})
\begin{align}
    \inf_{\calA:\R^{n\times n}\to\{0,1\}}
    \ppp{
        \P_{\calH_0}(\calA(\s{X})=1)
        +
        \P_{\calH_1}(\calA(\s{X})=0)}
    =
    1 - d_{\s{TV}}(\P_{\calH_0},\P_{\calH_1}),
\end{align}
and so detection is information-theoretically impossible whenever $d_{\s{TV}}(\P_{\calH_0},\P_{\calH_1}) = o(1)$. Our lower bounds depend on the following quantity
\begin{align}
    \Theta^\star
    \triangleq
    \max_{\ell \in [m]}
    \frac{1}{m^2 k^2}
    \log\p{
        \frac{1}{k^2}
        \sum_{u \in [k] \times [k]}
        \exp\p{
            m^2 k^2 \chi^2(\calP_{\ell,u}\Vert\calQ)}}.
    \label{eq:Theta_star_def}
\end{align}
The following result gives conditions under which $d_{\s{TV}}(\P_{\calH_0},\P_{\calH_1}) = o(1)$, thus precluding successful recovery. 
\begin{theorem}[Information-theoretic lower bounds]
\label{thrm:lower_bound}
Consider the finite-template submatrix detection model introduced in Section~\ref{sec:prob_form}. Let $\delta=\delta_n>0$ be any sequence such that $\delta_n\to 0$, as $n\to\infty$. The following statements hold.
\begin{enumerate}
    \item Consider the non-consecutive placement family $\calK_{k,m,n}$. If
    \begin{align}
        \Theta^\star \le \min\p{\frac{1}{k},\ \frac{n^2\log(1+\delta)}{2m^2k^4}},
        \label{eq:non_consecutive_impossibility}
    \end{align}
    then $d_{\s{TV}}(\P_{\calH_0},\P_{\calH_1})=o(1)$, and detection is information-theoretically impossible.
    \item Consider the circular consecutive placement family $\calK_{k,m,n}^{\circ}$. Assume $k\le \frac{n}{2}$. If
    \begin{align}
        \Theta^\star \le
        \frac{1}{k^2}\log\p{1+\frac{n^2\log(1+\delta)}{4k^2m^2}},
        \label{eq:consecutive_impossibility}
    \end{align}
    then $d_{\s{TV}}(\P_{\calH_0},\P_{\calH_1})=o(1)$, and detection is information-theoretically impossible.
\end{enumerate}
\end{theorem}

\begin{corollary}[Impossibility for standard consecutive placements]
\label{cor:standard_consecutive}
Consider the detection model in Section~\ref{sec:prob_form}, under the standard consecutive placement family $\calK_{k,m,n}^{\s{con}}$. Assume $mk=o(n)$. If \eqref{eq:consecutive_impossibility} holds, then $d_{\s{TV}}(\P_{\calH_0},\P_{\calH_1})=o(1)$, and detection is information-theoretically impossible.
\end{corollary}

\begin{proof}[Proof sketch of Corollary~\ref{cor:standard_consecutive}]
    The circular and standard consecutive placement models differ only through boundary effects. Under the circular model, block locations are translation invariant, whereas under the standard model only $n-k+1$ starting positions are allowed. For a uniformly random block under the circular model, the probability of wrapping around the boundary is $O(k/n)$. By a union bound over the $m$ planted blocks, the probability that at least one block wraps around is $O(mk/n)$. Hence, $d_{\s{TV}}(\P_{\calH_1}^{\circ},\P_{\calH_1}^{\s{con}}) = O(mk/n)$. In particular, if $mk=o(n)$ then $d_{\s{TV}}(\P_{\calH_1}^{\circ},\P_{\calH_1}^{\s{con}})=o(1)$, and the claim follows from Theorem~\ref{thrm:lower_bound}.
\end{proof}

To build intuition, we discuss the first few steps of the proof of Theorem~\ref{thrm:lower_bound}. We begin with the standard inequality $d_{\s{TV}}^2(P,Q) \leq \frac{1}{2}\chi^2(P\Vert Q)$, see, e.g.,~\cite[Sec.~2]{tsybakov2004introduction}. Thus, to obtain an impossibility result, it suffices to show that $\chi^2(\P_{\calH_1}\Vert\P_{\calH_0}) \to 0$. Due to the product structure under the null hypothesis and the finite-template construction under the alternative hypothesis, this task reduces to understanding how the \emph{entrywise} chi-square distances $\chi^2(\calP_{\ell,u}\Vert\calQ)$ accumulate in the second-moment calculation through overlaps between independently drawn planted configurations. This accumulation is, roughly speaking, captured by $\Theta^\star$. Indeed, for each template $\ell\in[m]$, \eqref{eq:Theta_star_def} aggregates the entrywise divergences $\chi^2(\calP_{\ell,u}\Vert\calQ)$ across template coordinates, and the maximum over $\ell$ corresponds to the template that yields the largest contribution to the second moment. As it turns out, $\Theta^\star$ characterizes the exponential growth rate of the second moment of the likelihood ratio and thus determines the information-theoretic detectability regime.

\subsection{Smooth-signal regime}

The lower and upper bounds in the previous subsections are general and hold for any set of templates. In this subsection, we show that for a non-trivial set of structured templates, these bounds align up to logarithmic factors in a specific regime. Specifically, for each template $\ell\in[m]$, let $\vartheta_\ell = (\vartheta_{\ell,u})_{u\in[k] \times [k]}$ denote a $k\times k$ array of local signal parameters. In the mean-shift model $\vartheta_{\ell,u}=(M_\ell)_u$, while in the variance-shift model $\vartheta_{\ell,u}=(\Sigma_\ell)_u$. Define the signal energy
\begin{align}
    \calE_\ell \triangleq \sum_{u\in[k] \times [k]} \vartheta_{\ell,u}^2,
    \qquad
    \calE \triangleq \max_{\ell\in[m]} \calE_\ell. \label{eq:energy_def}
\end{align}

\begin{definition}[Smooth-signal regime]
\label{def:smooth_signal}
We say that the finite-template model operates in the smooth-signal regime if, for all $\ell\in[m]$ and $u\in[k] \times [k]$:
\begin{enumerate}[label=(\roman*)]
    \item \label{item:i_boundness} Uniform boundedness: $\sup_{\ell\in[m]}\|\vartheta_\ell\|_\infty = O(1)$.
    \item \label{item:ii_non_spike} Non-spikiness:  $\sup_{\ell\in[m]} \frac{k^2\|\vartheta_\ell\|_\infty^2}{\calE_\ell} = O(1)$.
\end{enumerate}
\end{definition}
For this non-trivial family of templates, our upper bounds in Theorems~\ref{thrm:Detect_upper}--\ref{thrm:Detect_upper_variance} simplify as follows.

\begin{corollary}[Smooth-signal upper bounds] \label{cor:smooth_signal_upper}
Assume that the smooth-signal regime in Definition~\ref{def:smooth_signal} holds.
\begin{enumerate}[label=(\roman*)]
    \item \label{item:sum_test_corr} If 
    \begin{align}
         \calE=\omega\left(\frac{n^2}{m^2k^2}\right) \label{eq:mean_E_cond},
    \end{align}
   then the global sum test $\calA_{\s{sum}}(\s{X})$ in \eqref{eq:sum_algo} satisfies $\s{R}(\calA_{\s{sum}})=o(1)$ under the non-consecutive placement regime and under both consecutive placement regimes.
    \item \label{item:mean_scan_test} If 
    \begin{align}
        \calE=\omega(\log|\calB|),
    \end{align}
    then the template-aware scan test $\calA^{\mu}_{\s{scan, max}}(\s{X}; \calB)$ in     \eqref{eq:scan_algo} satisfies $\s{R}\p{\calA^{\mu}_{\s{scan, max}}}=o(1)$. In particular, for $\calB=\calB_{k,n}$ it suffices that $\calE=\omega\p{k\log\frac{n}{k}}$, and for $\calB\in\{\calB^{\s{con}}_{k,n},\calB^\circ_{k,n}\}$ it suffices that $\calE=\omega(\log n)$.
    \item \label{item:quad_test_corr} If
    \begin{align}
        \calE=\omega\p{\frac{n^2}{m^2k^2}} \label{eq:var_E_cond},
    \end{align}
    then the quadratic test $\calA_{\s{quad}}(\s{X})$ in \eqref{eq:sum_algo_var} satisfies $\s{R}(\calA_{\s{quad}})=o(1)$ under the non-consecutive placement regime and under both consecutive placement regimes.
     \item \label{item:var_scan_test} If 
     \begin{align}
         \calE=\omega(\log|\calB|+\log m),
     \end{align}
     then the finite-template scan test $\calA^{\sigma}_{\s{scan}}(\s{X};\calB)$ in \eqref{eq:scan_algo_var} satisfies $\s{R}\p{\calA^{\sigma}_{\s{scan}}}=o(1)$. In particular, for $\calB=\calB_{k,n}$ it suffices that $\calE=\omega\p{k\log\frac{n}{k}+\log m}$ and for $\calB\in\{\calB^{\s{con}}_{k,n},\calB^\circ_{k,n}\}$ it suffices that $\calE=\omega(\log n+\log m)$. 
\end{enumerate}
\end{corollary}

Note that Corollary~\ref{cor:smooth_signal_upper} implies that the global test outperforms the scan test whenever $\log|\calB| = o\p{\frac{n^2}{m^2k^2}}$, and this is independent of whether the placement is consecutive or not. Next, our lower bounds in Theorems~\ref{thrm:lower_bound} simplify as follows.
\begin{corollary}[Smooth-signal lower bounds]
\label{cor:smooth_signal}
Consider the smooth-signal regime in Definition~\ref{def:smooth_signal}.
\begin{enumerate}
    \item Under the non-consecutive placement family $\calK_{k,m,n}$, if
    \begin{align}
        \calE
        = o\p{
            k \ \wedge\ \frac{n^2}{m^2k^2}},
        \label{eq:cor_cond}
    \end{align}
    then $d_{\s{TV}}(\P_{\calH_0},\P_{\calH_1})=o(1)$.

    \item 
    Under the circular consecutive placement family $\calK_{k,m,n}^{\circ}$, if
    \begin{align}
        \calE
        = o\p{
            \log\p{1+\frac{n^2}{k^2m^2}}},
        \label{eq:cor_cond_con}
    \end{align}
    then $d_{\s{TV}}(\P_{\calH_0},\P_{\calH_1})=o(1)$.
\end{enumerate}
\end{corollary}

\begin{corollary}[Smooth-signal consecutive placements]
\label{cor:smooth_signal_consecutive}
Consider the smooth-signal regime in Definition~\ref{def:smooth_signal}. Under the standard consecutive placement family $\calK_{k,m,n}^{\s{con}}$. If \eqref{eq:cor_cond_con} and $mk=o(n)$ hold, then $d_{\s{TV}}(\P_{\calH_0},\P_{\calH_1})=o(1)$.
\end{corollary}


\begin{proof}[Proof of Corollary~\ref{cor:smooth_signal_consecutive}]
By Corollary~\ref{cor:smooth_signal}, condition~\eqref{eq:cor_cond_con}
implies $d_{\s{TV}}(\P_{\calH_0},\P_{\calH_1}^{\circ})=o(1)$ for the circular consecutive model.
If $mk=o(n)$, Corollary~\ref{cor:standard_consecutive} gives 
$d_{\s{TV}}(\P_{\calH_1}^{\circ},\P_{\calH_1}^{\s{con}})=o(1)$.
The result follows from the triangle inequality.
\end{proof}

\begin{table}[t] \centering
\begin{tabular}{lcccc}
\toprule
 & Lower bound 
 & Scan (mean) 
 & Scan (variance) 
 & Global tests \\
\midrule
Non-consecutive
& $k \wedge \dfrac{n^2}{m^2k^2}$
& $k\log(n/k)$
& $k\log(n/k)+\log m$
& $\dfrac{n^2}{m^2k^2}$ \\[1.2em]

Consecutive
& $\log\!\left(1+\dfrac{n^2}{k^2m^2}\right)$
& $\log n$
& $\log n + \log m$
& $\dfrac{n^2}{m^2k^2}$ \\
\bottomrule
\end{tabular}
\caption{Energy scales in the smooth-signal regime in Definition~\ref{def:smooth_signal}. The ``Lower bound'' column gives an impossibility condition: if $\calE=o(\cdot)$ then $d_{\s{TV}}(\P_{\calH_0},\P_{\calH_1})=o(1)$. The scan and global columns give sufficient conditions: if $\calE=\omega(\cdot)$ then the corresponding test attains vanishing risk.}
\label{tab:smooth_signal_scales}
\end{table}

Table~\ref{tab:smooth_signal_scales} summarizes the resulting energy scales and the resulting bounds as captured by Corollaries \ref{cor:smooth_signal_upper}--\ref{cor:smooth_signal_consecutive}. It is evident that the lower and upper bounds coincide up to poly-logarithmic factors.  
Finally, we note that the classical homogeneous mean-shift submatrix detection problem \cite{dadon2024detection} is a special case of this framework. Indeed, for homogeneous templates $M_\ell=\lambda\mathbb{I}_{k\times k}$, the signal energy satisfies $\calE=\|M_\ell\|_F^2=k^2\lambda^2$, and $\mu_{\s{det}}=\lambda mk^2$. In this setting, Corollary~\ref{cor:smooth_signal} yields the classical impossibility condition $|\lambda|=o\p{\sqrt{k}\wedge\frac{n}{mk^2}}$, while the first two items of Theorem~\ref{thrm:Detect_upper} show that detection is possible (using the global and scan detection algorithms) once $|\lambda|=\omega\p{\sqrt{\frac{\log(n/k)}{k}}\vee\frac{n}{mk^2}}$. These results coincide with \cite{dadon2024detection}.

\section{Proofs}

\subsection{Proof of Theorem \ref{thrm:Detect_upper}}
\subsubsection{Sum test}
\begin{proof}
We analyze the Type-I and Type-II errors for the sum statistic \eqref{eq:Sum_stat} and the test \eqref{eq:sum_algo}. 
Under $\calH_0$, the entries of $\s{X}$ are i.i.d.\ $\calN(0,1)$, so
\begin{align}
    \s{T}_{\s{sum}}(\s{X}) 
    = \s{sign}(\mu_{\s{det}}) \sum_{i,j \in[n]} \s{X}_{ij} 
    \sim \calN(0,n^2).
\end{align}
Hence for any $\tau \ge 0$  , 
\begin{align}
\label{eq:sum_typeI}
    \P_{\calH_0}( \mathcal{A}_{\s{sum}}(\s{X}) = 1 ) 
    &= \P_{\calH_0}(\s{T}_{\s{sum}}(\s{X}) \geq \tau ) \\
    &= \P(\calN(0, n^2) \geq \tau )\\
    &\leq \exp\p{-\frac{\tau^2}{2n^2}}.
\end{align}

Under $\calH_1$, the law of $\s{X}$ is the mixture over the random planted block collection $\s{K}$ and, conditional on $\s{K}$, the random labeling $\beta:\s{K} \to [m]$. 
Fix any realization $(\s{K}, \beta)$. For each planted block $\s{B} \in \s{K}$ and each
$(i,j)\in \s{B}$, we have 
\begin{align}
    \E\pp{\s{X}_{ij} \vert \s{K}, \beta} = \p{M_{\beta(\s{B})}}_{\varphi_{\s{B}}(i,j)}.
\end{align}
Since each entry of $M_{\beta(\s{B})}$ appears exactly once over the block $\s{B}$, it holds that 
\begin{align}
    \sum_{(i,j) \in \s{B}} \E\pp{\s{X}_{ij} \vert \s{K}, \beta} = \sum_{u,v \in [k]} \p{M_{\beta(\s{B})}}_{uv}
\end{align}
Summing over $\s{B} \in \s{K}$ and using that $\beta:\s{K}\to[m]$ is a bijection,
\begin{align}
    \sum_{\s{B} \in \s{K}} \sum_{(i,j)\in\s{B}}\E\pp{\s{X}_{ij} \vert \s{K}, \beta} = \sum_{\ell = 1}^m \sum_{u,v\in[k]}(M_\ell)_{uv} = \mu_{\s{det}}.
\end{align}
Therefore, 
\begin{align}
    \s{T}_{\s{sum}}(\s{X}) \vert (\s{K} , \beta) \sim \mathcal{N}\p{\abs{\mu_{\s{det}}}, n^2},
\end{align}
and hence $\s{T}_{\s{sum}}(\s{X})\sim \calN(\abs{\mu_{\s{det}}},n^2)$ under $\calH_1$. Therefore,
\begin{align}
    \P_{\calH_1}( \calA_{\s{sum}}(\s{X}) = 0 ) 
    &= \P_{\calH_1}(\s{T}_{\s{sum}}(\s{X}) \leq \tau) \\
    &= \P\p{\mathcal{N}\p{\abs{\mu_{\s{det}}}, n^2} \leq \tau}\\
    &\leq \exp\ppp{-\frac{(\tau - \abs{\mu_{\s{det}}})^2}{2n^2}}.
\end{align}
for any $\tau\le \abs{\mu_{\s{det}}}$. Choosing $\tau = \tau_{\s{sum}} = \abs{\mu_{\s{det}}}/2$ as in \eqref{eq:mean_thresholds}, yields $\s{R}(\mathcal{A}_{\s{sum}}) \le 2\exp\ppp{-\frac{\abs{\mu_{\s{det}}}^2}{8n^2}}$. In particular, if $\abs{\mu_{\s{det}}}/n \to \infty$, then $\s{R}(\calA_{\s{sum}}) = o(1)$. 

The argument does not depend on the placement family and therefore applies to the non-consecutive, standard consecutive, and circular consecutive models.
\end{proof}

\subsubsection{Scan test}
\begin{proof}
    Recall the scan statistic in \eqref{eq:Scan_stat} and the template-aware scan test in \eqref{eq:scan_algo}. In general, 
    \begin{align}
        \s{T}^{\mu}_{\s{scan}}(\s{X}; \calB, M) 
        = \max_{\s{B} \in \calB} \sum_{(i,j) \in \s{B}} M_{\varphi_{\s{B}}(i,j)} \s{X}_{ij}, \qquad
        \calA^{\mu}_{\s{scan, max}}(\s{X}) = \Ind\{\s{T}^{\mu}_{\s{scan}}(\s{X};\calB,M)\ge\tau\},
    \end{align}
    where $\calB=\calB_{k,n}$ or $\calB\in\{\calB_{k,n}^{\s{con}},\calB_{k,n}^{\circ}\}$.
    Let $M_{\max} = M_{\ell_{\max}}$ where $\ell_{\max}\in\arg\max_{\ell\in[m]}\|M_\ell\|_F^2$, 
    
    Under $\calH_0$, the entries of $\s{X}$ are i.i.d. $\calN(0,1)$. Hence for any fixed $\s{B} \in \calB$
    \begin{align}
        \sum_{(i,j) \in \s{B}} \p{M_{\max}}_{\varphi_{\s{B}}(i,j)} \s{X}_{ij} \sim \calN(0, \norm{M_{\max}}_{F}^2),
    \end{align}
     since $\varphi_{\s{B}}$ maps the indices of $\s{B}$ to $[k]\times[k]$.      Therefore, applying the union bound over $\s{B} \in \calB$ and a Gaussian tail bound yields
    \begin{align}
        \P_{\calH_0}\p{\calA^{\mu}_{\s{scan, max}}(\s{X}) =1} 
        &= 
        \P_{\calH_0}\p{\s{T}^{\mu}_{\s{scan}}(\s{X}; \calB, M_{\max}) \ge \tau}
        \\
        &\le  \sum_{\s{B} \in \calB} \P\p{\calN(0 , \norm{M_{\max}}_{F}^2) \ge \tau}\\
        \label{eq:Gaus_chernoff}
        &\le \abs{\calB} \exp\ppp{-\frac{\tau^2}{2\norm{M_{\max}}_{F}^2}}. 
    \end{align}

    Under $\calH_1$, the law of $\s{X}$ is the mixture over the random planted block collection $\s{K}^{\star}$ and, conditional on $\s{K}^{\star}$, the random labeling $\beta:\s{K}^{\star} \to [m]$. Fix any realization $(\s{K}^{\star}, \beta)$. Define the planted block carrying $M_{\s{max}}$ by $\s{B}_{\max}^{\star} \triangleq \beta^{-1}(\ell_{\max})$ which is well-defined and unique since $\beta$ is a bijection. 
    
    By definition of the scan statistic as a maximum over $\calB$, we have
    \begin{align}
        \s{T}^{\mu}_{\s{scan}}(\s{X}; \calB, M_{\max}) \ge
        \sum_{(i,j)\in \s{B}_{\max}^\star} (M_{\max})_{\varphi_{\s{B}_{\max}^\star}(i,j)} \s{X}_{ij}. 
    \end{align}
    Conditionally on $(\s{K}^\star, \beta)$, entries $(i,j)$ on $\s{B}_{\max}^\star$ are independent and satisfy    \begin{align}
        \s{X}_{ij} = (M_{\max})_{\varphi_{\s{B}^\star_{\max}}(i,j)} + \s{Z}_{ij}, \quad \text{where} \quad \s{Z}_{ij}\stackrel{i.i.d.}{\sim}\calN(0,1). 
    \end{align}
    Therefore,
    \begin{align}
        \sum_{(i,j) \in \s{B}_{\max}^\star} (M_{\max})_{\varphi_{\s{B}_{\max}^\star}(i,j)} \s{X}_{ij}
        &= \sum_{(i,j)\in\s{B}_{\max}^\star} (M_{\max})_{\varphi_{\s{B}_{\max}^\star}(i,j)}^2 +  \sum_{(i,j)\in\s{B}_{\max}^\star}(M_{\max})_{\varphi_{\s{B}_{\max}^\star}(i,j)} \s{Z}_{ij} \\
        &= \norm{M_{\max}}_{F}^2 + \calN\p{0, \norm{M_{\max}}_{F}^2} \\
        &\sim \calN\p{\norm{M_{\max}}_{F}^2, \norm{M_{\max}}_{F}^2}. \label{eq:mean_scan_typeI}
    \end{align}
    It follows that, conditional on $(\s K^\star,\beta)$, for $\tau \le |M_{\max}|_F^2$ a Gaussian lower-tail bound gives
    \begin{align}
        \P_{\calH_1}\p{\calA^{\mu}_{\s{scan, max}}(\s{X}) = 0 \,\vert\, \s{K}^\star, \beta}  
        &= 
        \P_{\calH_1}\p{\s{T}^{\mu}_{\s{scan}}(\s{X}; \calB, M_{\max}) \le \tau \vert \s{K}^\star, \beta}
        \\
        &\le  \P\p{ \calN\p{\norm{M_{\max}}_{F}^2, \norm{M_{\max}}_{F}^2} \le \tau}\\
        &\le \exp\ppp{-\frac{\p{\tau - \norm{M_{\max}}_{F}^2}^2}{2 \norm{M_{\max}}_{F}^2}} \label{eq:scan_typeII_bound_cond}.
\end{align}
Since the bound in \eqref{eq:scan_typeII_bound_cond} does not depend on $(\s{K}^\star, \beta)$,
it also holds for the marginal Type-II error under $\calH_1$, that is
\begin{align}
    \P_{\calH_1}\p{\calA^{\mu}_{\s{scan, max}}(\s{X}) = 0} \le \exp\ppp{-\frac{\p{\tau - \norm{M_{\max}}_{F}^2}^2}{2 \norm{M_{\max}}_{F}^2}}. \label{eq:mean_scan_typeII}
\end{align}
Combining~\eqref{eq:mean_scan_typeI} and~\eqref{eq:mean_scan_typeII} yields
\begin{align}
    \s{R}\p{\calA^{\mu}_{\s{scan, max}}(\s{X})} \le {\abs{\calB} \exp\ppp{-\frac{\tau^2}{2\norm{M_{\max}}_{F}^2}} + \exp\ppp{-\frac{\p{\tau - \norm{M_{\max}}_{F}^2}^2}{2 \norm{M_{\max}}_{F}^2}}}.
\end{align}
Finally, note that $\abs{\calB_{k,n}} = \binom{n}{k}^2 \le \p{\frac{en}{k}}^{2k}$ and $\abs{\calB_{k,n}^{\s{con}}} = (n-k+1)^2 \le n^2$, with the same bound holding for $\calB_{k,n}^{\circ}$.
Substituting the corresponding bound on $\abs{\calB}$ and choosing the $\tau$ as in  \eqref{eq:mean_thresholds} yield $\s{R}\p{\calA^{\mu}_{\s{scan, max}}(\s{X})} = o(1)$ under the conditions stated in items~\ref{item:Upper_mean_scan_non_con}-\ref{item:Upper_mean_scan_con} of Theorem \ref{thrm:Detect_upper}. The argument depends on $\calB$ only through its cardinality and therefore applies to the non-consecutive, standard consecutive, and circular consecutive placement models.
\end{proof}

\subsection{Proof of Theorem \ref{thrm:Detect_upper_variance}}
\subsubsection{Quadratic test}
\begin{proof}
    Recall the global centered quadratic statistic $\s{T}_{\s{quad}}$ in \eqref{eq:Sum_stat_var}, and the quadratic test $\calA_{\s{quad}}$ in \eqref{eq:sum_algo_var} with the threshold $\tau_{\s{quad}}=\nu_{\s{det}}/2$ as in \eqref{eq:variance_thresholds}. We bound the Type-I and Type-II error probabilities. 
    
    Under $\calH_0$, $\s{X}_{ij} \sim \calN(0,1)$, hence $\E_{\calH_0}\pp{\s{X}_{ij}^2 - 1} = 0$ and $\Var_{\calH_0}(\s{X}_{ij}^2 - 1) = 2$. Therefore, $\Var_{\calH_0}(\s{T}_{\s{quad}}) = 2n^2$. 
    By Chebyshev's inequality,
    \begin{align}
        \P_{\calH_0}\p{\calA_{\s{quad}}(\s{X})=1}
        &= \P_{\calH_0}\p{\s{T}_{\s{quad}}(\s{X})\ge \tau_{\s{quad}}}
        \\
        &\le
        \frac{\Var_{\calH_0}\p{\s{T}_{\s{quad}}(\s{X})}}{\tau_{\s{quad}}^2}
        =
        \frac{2n^2}{(\nu_{\s{det}}/2)^2}
        =
        \frac{8n^2}{\nu_{\s{det}}^2}
        = o(1),
    \end{align}
    whenever $\nu_{\s{det}} = \omega(n)$. 

    Under $\calH_1$, conditionally on $(\s{K}, \beta)$, each planted entry satisfies $\s{X}_{ij}\sim\calN\p{0,\,1+(\Sigma_{\beta(\s{B})})_{\varphi_{\s{B}}(i,j)}}$ for $(i,j)\in\s{B}$ and $\s{X}_{ij}\sim\calN(0,1)$ otherwise. Therefore,
    \begin{align}
        \E_{\calH_1}\pp{\s{X}_{ij}^2-1 \vert \s{K},\beta}
        =
        \begin{cases}
            (\Sigma_{\beta(\s{B})})_{\varphi_{\s{B}}(i,j)}, & (i,j)\in\s{B}\in\s{K},\\
            0, & \text{otherwise},
        \end{cases}
    \end{align}
    and summing over all entries yields
    \begin{align}
        \E_{\calH_1}\pp{\s{T}_{\s{quad}}(\s{X}) \vert \s{K},\beta}
        =
        \sum_{\s{B}\in\s{K}}\sum_{(i,j)\in\s{B}} (\Sigma_{\beta(\s{B})})_{\varphi_{\s{B}}(i,j)}
        =
        \nu_{\s{det}}.
    \end{align}\allowdisplaybreaks
    Moreover, since the entries are independent conditional on $(\s{K},\beta)$ and for $Y\sim\calN(0,\sigma^2)$ we have $\Var(Y^2-1)=2\sigma^4$,
    \begin{align}
        \Var_{\calH_1}\p{\s{T}_{\s{quad}}(\s{X}) \vert \s{K},\beta}
        &=
        \sum_{i,j \in [n]}
        \Var_{\calH_1}\p{\s{X}_{ij}^2 - 1 \vert \s{K},\beta}
        \\
        &=
        \sum_{(i,j) \notin \bigcup_{\s{B}\in\s{K}} \s{B}}
        \Var\p{\s{X}_{ij}^2 - 1}
        +
        \sum_{\s{B}\in\s{K}}
        \sum_{(i,j)\in\s{B}}
        \Var_{\calH_1}\p{\s{X}_{ij}^2 - 1 \vert \s{K},\beta}
        \\
        &=
        2(n^2 - mk^2)
        +
        \sum_{\s{B}\in\s{K}}
        \sum_{(i,j)\in\s{B}}
        2\Var\p{\s{X}_{ij} \vert \s{K},\beta}^2
        \\
        &=
        2(n^2 - mk^2)
        +
        2\sum_{\s{B}\in\s{K}}
        \sum_{(i,j)\in\s{B}}
        \p{1 + (\Sigma_{\beta(\s{B})})_{\varphi_{\s{B}}(i,j)}}^2
        \\
        &=
        2(n^2 - mk^2)
        +
        2\sum_{\ell=1}^m
        \sum_{u,v\in[k]}
        \p{1 + (\Sigma_\ell)_{uv}}^2
        \\
        &=
        2n^2
        + 4 \nu_{\s{det}}
        + 2\sum_{\ell=1}^m
        \sum_{u,v\in[k]}
        (\Sigma_\ell)_{uv}^2 .
    \end{align}
    Since each template appears exactly once, both $\E_{\calH_1}\pp{\s{T}_{\s{quad}}(\s{X}) \vert \s{K},\beta}$ and $\Var_{\calH_1}\p{\s{T}_{\s{quad}}(\s{X}) \vert \s{K},\beta}$ are deterministic (they depend only on the template family), hence they equal the corresponding unconditional quantities. Applying Chebyshev's inequality with $\tau_{\s{quad}}=\nu_{\s{det}}/2$ gives
    \begin{align}
        \P_{\calH_1}\p{\calA_{\s{quad}}(\s{X})=0}
        &=
        \P_{\calH_1}\p{\s{T}_{\s{quad}}(\s{X}) \le \tau_{\s{quad}}}
        \\
        &\le
        \frac{\Var_{\calH_1}\p{\s{T}_{\s{quad}}(\s{X})}}{(\E_{\calH_1}[\s{T}_{\s{quad}}(\s{X})]-\tau_{\s{quad}})^2}
        \\
        &=
        \frac{2n^2+4\nu_{\s{det}}+2\sum_{\ell=1}^m\sum_{u,v\in[k]}(\Sigma_\ell)_{uv}^2}{(\nu_{\s{det}}/2)^2}.
    \end{align}
    Using $\max_{\ell,u,v}(\Sigma_\ell)_{uv}\le \vartheta_0$ we have
    \begin{align}
        \sum_{\ell=1}^m\sum_{u,v\in[k]}(\Sigma_\ell)_{uv}^2
        \le
        \vartheta_0\,\nu_{\s{det}},
    \end{align}
    The numerator is $O(n^2+\nu_{\s{det}})$, so the bound is $o(1)$ whenever $\nu_{\s{det}}=\omega(n)$.
    The argument does not depend on the placement family and therefore applies to the non-consecutive, standard consecutive, and circular consecutive models.
\end{proof}

\subsubsection{Scan test}
\begin{proof}
    Recall the scan statistic in \eqref{eq:Scan_stat_HVar} and the finite-template scan test $\calA^{\sigma}_{\s{scan}}(\s{X};\calB)$ in \eqref{eq:scan_algo_var}, where $\calB=\calB_{k,n}$ for non-consecutive placements, $\calB=\calB^{\s{con}}_{k,n}$ for standard consecutive placements, and $\calB=\calB^{\circ}_{k,n}$ for circular consecutive placements.
       
    For any $k \times k$ block $\s{B} \in \calB$, we define the following distributions:
    \begin{itemize}
        \item Null distribution. We define $\calP_0^{(\s{B})} \triangleq \otimes_{(i,j)\in \s{B}}\calN(0,1)$ representing the joint distribution of the entries in $\s{B}$ under the null hypothesis.
        \item Local alternative. For a given template $\Sigma_{\ell} \in \calS$, we define $\calP_\ell^{(\s{B})} \triangleq \otimes_{(i,j)\in\s{B}} \calN(0, 1 + (\Sigma_{\ell})_{\varphi_{\s{B}}(i,j)})$. This distribution is not the true law under  $\calH_1$; rather, it represents a local alternative obtained by embedding the template $\Sigma_{\ell}$ into the block $\s{B}$.
    \end{itemize}
    Accordingly, we define the $\ell$-template-matched loglikelihood score of $\s{B}$
    \begin{align}
        \calL_{\ell}(\s{B})  \triangleq \log \frac{\calP_{\ell}^{(\s{B})}}{\calP_{0}^{(\s{B})}}(\s{X}) = \sum_{(i,j) \in \s{B}} \frac{1}{2} \p{\frac{(\Sigma_{\ell})_{\varphi_{\s{B}}(i,j)}}{1 + (\Sigma_{\ell})_{\varphi_{\s{B}}(i,j)}} \s{X}_{ij}^2 - \log\p{1 + (\Sigma_{\ell})_{\varphi_{\s{B}}(i,j)}}},
    \end{align}
    and the corresponding blockwise KL-divergence 
    \begin{align}
        \s{KL}_{\ell}(\s{B}) \triangleq \s{d}_{\s{KL}}(\calP_{\ell}^{(\s{B})}\Vert \calP_{0}^{(\s{B})}) = \frac{1}{2} \sum_{u,v \in [k]} \p{(\Sigma_{\ell})_{uv} - \log(1 + (\Sigma_{\ell})_{uv})},
    \end{align}
    which does not depend on $\s{B}$. 
    The finite-template scan statistic is
    \begin{align}
        \s{T}^{\sigma}_{\s{scan}}(\s{X};\calB)
        \triangleq
        \max_{\ell\in[m]} \max_{\s{B}\in \calB} \calL_{\ell}(\s{B}).
    \end{align}   
    We analyze the Type-I and Type-II error probabilities. 
    Under $\calH_0$, the entries $\s{X}_{ij}$ are i.i.d. $\calN(0,1)$. For any fixed $(\ell,\s{B})$, $\calL_{\ell}(\s{B})$ is a log-likelihood ratio, hence $\E_{\calH_0}\pp{e^{\calL_{\ell}(\s{B})}} = 1$. By Markov's inequality,
    \begin{align}
        \P_{\calH_0}\p{\calL_{\ell}(\s{B}) \ge \tau^{\sigma}_{\s{scan}}} 
        \le 
        e^{-\tau^{\sigma}_{\s{scan}}}.
    \end{align}
    Applying a union bound over all $\ell\in[m]$ and $\s{B}\in\calB$,     \begin{align}
        \P_{\calH_0}\p{\s{T}^{\sigma}_{\s{scan}}(\s{X};\calB) \ge \tau^{\sigma}_{\s{scan}}} 
        \le \sum_{\ell=1}^{m} \sum_{\s{B} \in \calB} \P_{\calH_0}\p{\calL_{\ell}(\s{B}) \ge \tau^{\sigma}_{\s{scan}}} 
        \le m \abs{\calB} e^{-\tau^{\sigma}_{\s{scan}}} 
    \end{align}
    Setting $\tau^{\sigma}_{\s{scan}} = (1+\delta)(\log|\calB| + \log m)$ yields $\P_{\calH_0}\p{ \s{T}^{\sigma}_{\s{scan}}(\s{X};\calB) \ge \tau^{\sigma}_{\s{scan}}} \le (m|\calB|)^{-\delta} = o(1)$.

    Under $\calH_1$, fix a realization $(\s{K}^{\star}, \beta)$. Let $\ell^\star \in \arg\max_{\ell\in[m]} \s{KL}(\Sigma_\ell)$, where $\s{KL}(\Sigma)$ is defined in~\eqref{eq:KL_definition}, and define $\s{B}^\star \triangleq \beta^{-1}(\ell^\star)$. Since the scan ranges over all templates and all blocks,
    \begin{align}
        \s{T}^{\sigma}_{\s{scan}}(\s{X};\calB) \ge \calL_{\ell^\star}(\s{B}^\star).
    \end{align}

    Conditionally on $(\s{K}^{\star}, \beta)$, the entries on $\s{B}^{\star}$ 
    satisfy
    \begin{align}
        \s{X}_{ij}
        = \sqrt{1 + (\Sigma_{\ell^\star})_{\varphi_{\s{B}^\star}(i,j)}}\,\s{Z}_{ij}, \qquad 
        \s{Z}_{ij}\sim\calN(0,1).
    \end{align}
    Let $(\Sigma_{\ell^\star})_{uv} \triangleq \sigma_{uv}$. Then, substituting into the definition of $\calL_{\ell^{\star}}(\s{B}^{\star})$ yields
    \begin{align}
        \calL_{\ell^{\star}}(\s{B}^{\star}) &= \sum_{(i,j) \in \s{B}^{\star}} \frac{1}{2} \p{(\Sigma_{\ell^{\star}})_{\varphi_{\s{B}^{\star}}(i,j)} \s{Z}_{ij}^2 - \log\p{1 + (\Sigma_{\ell^{\star}})_{\varphi_{\s{B}^{\star}}(i,j)}}}\\
        &= \sum_{u,v \in [k]} \frac{1}{2} \p{\sigma_{uv} \s{Z}_{uv}^2 - \log\p{1 + \sigma_{uv}}}\\
        &= \frac{1}{2} \sum_{u,v \in [k]} \p{\sigma_{uv} - \log\p{1 + \sigma_{uv}}} + \frac{1}{2} \sum_{(u,v) \in [k]} \sigma_{uv} \p{\s{Z}_{uv}^2 - 1}, \label{eq:pm_1}\\
        &= \s{KL}(\Sigma_{\ell^\star}) + \frac{1}{2} \sum_{(u,v) \in [k]} \sigma_{uv} \p{\s{Z}_{uv}^2 - 1} 
        .
    \end{align}

    The distribution of $\calL_{\ell^{\star}}(\s{B}^\star)$ depends only on the template $\Sigma_{\ell^\star}$ and not on the remainder of $(\s{K}^{\star},\beta)$; hence the conditional and unconditional probabilities coincide. Therefore,
    \begin{align}
        \P_{\calH_1}\p{\calA^{\sigma}_{\s{scan}}(\s{X};\calB)=0}
        &= \P_{\calH_1}(\s{T}^{\sigma}_{\s{scan}}(\s{X}; \calB) \le \tau^{\sigma}_{\s{scan}})\\
        &\le \P_{\calH_1}\pp{\calL_{\ell^{\star}}(\s{B}^{\star}) < \tau^{\sigma}_{\s{scan}}}\\
        &= \P_{\calH_1}\pp{\sum_{u,v \in [k]} \sigma_{uv}(\s{Z}_{uv}^2 - 1) < 2\p{\tau^{\sigma}_{\s{scan}} -\s{KL}(\Sigma_{\max})} }.
    \end{align}
    Let $\Delta \triangleq \tau^{\sigma}_{\s{scan}} -\s{KL}(\Sigma_{\max})$. For any $\lambda>0$, applying Chernoff's bound gives
    \begin{align}
        \P_{\calH_1}\p{\calA^{\sigma}_{\s{scan}}(\s{X};\calB)=0} 
        &\le e^{-2\lambda \Delta} \E\pp{e^{-\lambda \sum_{u,v \in [k]}\sigma_{uv}(\s{Z}_{uv}^2 - 1)}}\\
        &= e^{-2\lambda \Delta} \prod_{u,v \in [k]} \E\pp{e^{-\lambda \sigma_{uv}(\s{Z}_{uv}^2 - 1)}}\\
        &= \exp \ppp{-2\lambda \Delta + \lambda^2 \sum_{u,v \in [k]} \sigma_{uv}^2} \\
        &= \exp \ppp{-2\lambda \Delta + \lambda^2 \norm{\Sigma_{\ell^{\star}}}_F^2}, \label{eq:var_typeII_chrnoff_lambda}
    \end{align}
    where we use the fact that $\ppp{\s{Z}_{uv}}_{u,v \in [k]}$ are i.i.d. standard Gaussian with moment generating function $\E\pp{e^{t\s{Z}^2}} = (1 - 2t)^{-1/2}$. For $\lambda > 0$ satisfying $2\lambda \sigma_{uv} < 1$, we obtain     $\E\pp{e^{-\lambda \sigma_{uv}(\s{Z}_{uv}^2 - 1)}} = e^{\lambda \sigma_{uv}} (1 + 2\lambda\sigma_{uv})^{-1/2}$. Taking $\log$ yields
    \begin{align}
        \log \E\pp{e^{-\lambda \sigma_{uv}(\s{Z}_{uv}^2 - 1)}} &= \lambda \sigma_{uv} -\frac{1}{2} \log(1 + 2\lambda\sigma_{uv})
        \le 
        \lambda^2 \sigma_{uv}^2,
    \end{align}
    where we used $\log(1 + x) \ge x - \tfrac{x^2}{2}$ for all $x \ge 0$.
    Finally, $\E\pp{e^{-\lambda \sigma_{uv}(\s{Z}_{uv}^2 - 1)}} \le \exp\ppp{\lambda^2 \sigma_{uv}^2}$. 
    Since $0 \le \sigma_{uv} \le \vartheta_0 < 1$, the admissibility condition $2\lambda \sigma_{uv} < 1$ holds uniformly whenever $\lambda < 1/(2\vartheta_0)$.
    
    We now minimize the right-hand side of \eqref{eq:var_typeII_chrnoff_lambda} over $\lambda$. Completing the square,
    \begin{align}
        - 2 \lambda \Delta + \lambda^2 \norm{\Sigma_{\ell^{\star}}}_F^2 = \norm{\Sigma_{\ell^{\star}}}_F^2 \p{\lambda - \frac{\Delta}{\norm{\Sigma_{\ell^{\star}}}_F^2}}^2 - \frac{\Delta^2}{\norm{\Sigma_{\ell^{\star}}}_F^2},
    \end{align}
    so the minimum is achieved at $\lambda^\star = \frac{\Delta}{\norm{\Sigma_{\ell^{\star}}}_F^2}$. 
    Under the corresponding assumption of Theorem~\ref{thrm:Detect_upper_variance} (items~\ref{item:variance_scan_non_con}-~\ref{item:variance_scan_con}), we have $\max_{\ell} \s{KL}(\Sigma_\ell) \gg \log|\calB| + \log m$. Since $\tau^{\sigma}_{\s{scan}} = (1 + \delta)\p{\log \abs{\calB} + \log m}$, it follows that $\Delta = \KL(\Sigma_{\ell^\star}) - \tau^{\sigma}_{\s{scan}} \to +\infty$ and in particular $0<\Delta<\s{KL}(\s{\Sigma}_{\ell^{\star}})$ for all sufficiently large $n$. The minimizer $\lambda^\star = \frac{\Delta}{\norm{\Sigma_{\ell^\star}}_F^2}$ therefore satisfies $\lambda^\star>0$. Moreover, since $\log(1+x)\ge x-\tfrac{x^2}{2}$ for $x\ge0$, we have for all $\Sigma \in \calS$
    \begin{align}
        \s{KL}(\Sigma) = \frac{1}{2}\sum_{u,v\in[k]}\p{\sigma_{uv}-\log(1+\sigma_{uv})}\le \frac{1}{4}\sum_{u,v\in[k]}\sigma_{uv}^2 = \frac{1}{4}\norm{\Sigma}_F^2,
    \end{align}
    hence $\lambda^\star \le \s{KL}(\Sigma_{\ell^\star})/\norm{\Sigma_{\ell^\star}}_F^2 \le 1/4$.
    Because $0\le \sigma_{uv}\le \vartheta_0<1$, we have $1/(2\vartheta_0)>1/2$, and therefore $1/4<1/(2\vartheta_0)$. Thus $\lambda^\star < 1/(2\vartheta_0)$, and the chosen optimizer lies within the range where the moment generating function bound is valid.
    
    Substituting $\lambda^\star$ 
    yields
    \begin{align}
        \P_{\calH_1}\p{\calA^{\sigma}_{\s{scan}}(\s{X};\calB)=0} \le \exp\ppp{-\frac{\p{\s{KL}(\Sigma_{\ell^{\star}}) -\tau^{\sigma}_{\s{scan}}}^2}{ \norm{\Sigma_{\ell^{\star}}}_{F}^2}}.
    \end{align}
    Choosing $\tau^{\sigma}_{\s{scan}} = (1 + \delta)\p{\log \abs{\calB} + \log m}$ implies $\P_{\calH_1}\p{\calA^{\sigma}_{\s{scan}}(\s{X};\calB)=0} = o(1)$ whenever $\s{KL}\p{\Sigma_{\ell^{\star}}} = \max_{\ell\in [m]} \s{KL}(\Sigma_{\ell}) \gg \log \abs{\calB} + \log m$. Hence, under this condition, the Type-II error probability vanishes.

    The argument applies uniformly to all placement families; the dependence on the placement model enters only through the cardinality $|\calB|$. 
    For non-consecutive placements, $|\calB_{k,n}| = \binom{n}{k}^2 \le (en/k)^{2k}$. For standard consecutive placements, $|\calB^{\s{con}}_{k,n}| = (n-k+1)^2$, and for circular consecutive placements, $|\calB^{\circ}_{k,n}| = n^2$. Substituting these bounds into threshold $\tau_{\s{scan}}=(1+\delta)(\log|\calB|+\log m)$ and the condition above 
    yields the thresholds and regimes stated in the main results.
\end{proof}

\subsection{Proof of Corollary \ref{cor:smooth_signal_upper}
}
We treat each item separately. 
\paragraph{~\ref{item:sum_test_corr} Mean-shift (global sum test).} 
Recall $\mu_{\det}
= \sum_{\ell=1}^m \sum_{u,v\in[k]} (M_\ell)_{uv}$. For any matrix $A\in\mathbb{R}^{k\times k}$, it holds that $\abs{\sum_{u,v \in [k]} A_{uv}}
\le \sum_{u,v \in [k]} |A_{uv}|
\le k \norm{A}_F$. Hence, 
\begin{align}
    |\mu_{\s{det}}| \le \sum_{\ell=1}^m \abs{\sum_{u,v \in [k]} (M_\ell)_{uv}} 
    \le k \sum_{\ell=1}^m \norm{M_\ell}_F
    \le m k \max_{\ell \in [m]} \norm{M_\ell}_F. 
\end{align}
Since $\calE=\max_{\ell \in [m]}\norm{M_\ell}_F^2$, we obtain $ |\mu_{\det}| \le mk \sqrt{\calE}$. Therefore, if $|\mu_{\det}|=\omega(n)$, then necessarily 
\begin{align}
    mk\sqrt{\calE}=\omega(n),
\end{align}
which implies~\eqref{eq:mean_E_cond}. The risk statement follows directly from item~\ref{item:mean_upper_sum} in  Theorem~\ref{thrm:Detect_upper}.

\paragraph{~\ref{item:mean_scan_test} Mean-shift (template-awere scan test).} 
In the mean-shift model, $\calE_{\ell}=\norm{M_\ell}_F^2$, and $\calE=\|M_{\max}\|_F^2$. Thus, the scan condition of Theorem~\ref{thrm:Detect_upper} is equivalent to $\calE=\omega(\log|\calB|)$. For $\calB=\calB_{k,n}$, $\log|\calB|=\Theta\p{k\log\frac{n}{k}}$, and for $\calB\in\ppp{\calB^{\s{con}}_{k,n},\calB^\circ_{k,n}}$, $\log|\calB|=\Theta(\log n)$. The risk bound follows from Theorem~\ref{thrm:Detect_upper}.

\paragraph{~\ref{item:quad_test_corr} Variance-shift (global quadratic test).} 
Define, for each $\ell\in[m]$, $S_\ell=\sum_{u,v \in [k]}(\Sigma_\ell)_{uv}$ and $\calE_\ell=\norm{\Sigma_\ell}_F^2$, so that $\nu_{\s{det}}=\sum_{\ell=1}^m S_\ell$ and $\calE=\max_{\ell} \calE_\ell$. 
Since $(\Sigma_\ell)_{uv} \ge 0$ for all $u,v \in [k]$, Cauchy--Schwarz yields, for every $\ell$,
\begin{align}
    S_\ell
    \le k \norm{\Sigma_\ell}_F
    = k\sqrt{\calE_\ell}. \label{eq:S_upper_bound}
\end{align}
Therefore,
\begin{align}
    \nu_{\det}=\sum_{\ell=1}^m S_\ell \le mk\sqrt{\calE}.
\end{align}
If the global quadratic test condition~\eqref{eq:quad_cond} holds, i.e. $\nu_{\s{det}}=\omega(n)$, then $mk\sqrt{\calE}=\omega(n)$, and hence
\begin{align}
    \calE=\omega\p{\frac{n^2}{m^2k^2}},
\end{align}
which proves~\eqref{eq:var_E_cond}. The risk bound $\s{R}(\calA_{\s{quad}})=o(1)$ follows from
Theorem~\ref{thrm:Detect_upper_variance}.

\paragraph{~\ref{item:var_scan_test} Variance-shift (finite-template scan).} 
Recall that 
\begin{align}
    \s{KL}(\Sigma_\ell) = \frac{1}{2} \sum_{u,v\in[k]} \p{ (\Sigma_\ell)_{uv} - \log(1+(\Sigma_\ell)_{uv})}.
\end{align}
Under the bounded-variance assumption, $0 \le (\Sigma_\ell)_{uv} \le \vartheta_0 < 1$. For $0\le x\le \vartheta_0$, define $f(x)\triangleq x-\log(1+x)$. Then $f(0)=f'(0)=0$ and $f''(x)=\frac{1}{(1+x)^2} \in \pp{\frac{1}{(1+\vartheta_0)^2}, 1}$. By Taylor's theorem, for each $x\in[0,\vartheta_0]$,
\begin{align}
    \frac{1}{2(1+\vartheta_0)^2}x^2 
    \le f(x) 
    \le \frac{1}{2} x^2.
\end{align}
Applying this bound entrywise yields
\begin{align}
    \frac{1}{4(1+\vartheta_0)^2} \norm{\Sigma_\ell}_F^2
    \le \s{KL}(\Sigma_\ell)
    \le \frac{1}{4} \norm{\Sigma_\ell}_F^2.
\end{align}
Hence,
\begin{align}
    \max_{\ell\in[m]}\s{KL}(\Sigma_\ell) =\Theta\p{\max_{\ell\in[m]} \norm{\Sigma_\ell}_F^2} 
    =\Theta(\calE).
\end{align}
The scan condition in Theorem~\ref{thrm:Detect_upper_variance} is therefore equivalent to
\begin{align}
    \calE=\omega(\log|\calB|+\log m),
\end{align}
and the risk bound follows from Theorem~\ref{thrm:Detect_upper_variance}.

\subsection{Proof of Theorem \ref{thrm:lower_bound}
}

We use a single second-moment bound that applies to both the non-consecutive and circular consecutive models, and then plug in the corresponding overlap estimates for each placement family.

\paragraph{Likelihood ratio and second moment.}
In order to lower bound the optimal risk, we apply the second-moment method, which reduces the problem to bounding the second moment of the likelihood ratio under $\calH_0$. In particular, 
\begin{align}
    d_{\s{TV}}(\P_{\calH_0}, \P_{\calH_1}) \le \frac{1}{2} \sqrt{\chi^2(\P_{\calH_1} \Vert \P_{\calH_0})} = \frac{1}{2} \sqrt{\E_{\calH_0}\pp{\s{L}_n(\s{X})^2} - 1}.
\end{align}
Thus, it suffices to show that $\E_{\calH_0}\pp{\s{L}n(\s{X})^2}\le 1+o(1)$, which implies $\chi^2(\P{\calH_1}\Vert\P_{\calH_0})=o(1)$ and hence $d_{\s{TV}}(\P_{\calH_0},\P_{\calH_1})=o(1)$.

Recall that under $\calH_0$, the entries $\{\s{X}_{ij}\}_{i,j \in [n]}$ are independent with common density $\calQ$, that is $\P_{\calH_0} = \calQ^{\otimes n^2}$. Under $\calH_1$, we draw a random set $\s{K}\in \calK$ of $m$ disjoint blocks (either $\calK_{k,m,n}$, $\calK_{k,m, n}^{\s{con}}$ or $\calK_{k,m,n}^{\circ}$as defined in~\eqref{eq:block_set_con}-\eqref{eq:block_set_circ}). 
Given $\s{K}$, a uniform random bijection $\beta: \s{K} \to [m]$ assigns a label to each block, so that $\beta(\s{B}) = \ell$ indicates that block $\s{B}$ carries template $\ell$. Thus,
\begin{align}
    \P_{\calH_1} = \E_{\s{K}, \beta}\pp{\P_{\s{K}, \beta}}.
\end{align}
Conditional on $(\s{K}, \beta)$, the entries remain independent, and for each $(i,j) \in [n]^2$,
\begin{align}
    \s{X}_{ij} \sim 
    \begin{cases}
        \calP_{\beta(\s{B}), \varphi_{\s{B}}(i,j)}, 
        & \textnormal{if } (i,j) \in \s{B} \text{ for some } \s{B} \in \s{K},\\
        \calQ, 
        & \textnormal{otherwise}.
    \end{cases}
\end{align}
Here, for each block $\s{B}\in\s{K}$, 
$\varphi_{\s{B}}: \s{B} \to [k] \times [k]$ denotes the deterministic coordinate map defined in \eqref{eq:induced_coord_map}, and $\{\calP_{\ell,u}\}_{\ell \in [m],\, u \in [k] \times [k]}$ is the unified family of signal densities introduced in Section~\ref{sec:prob_form}.
We assume that
\begin{align}
    \calP_{\ell, u} \ll \calQ,
    \qquad 
    \textnormal{for all } \ell \in [m],\ u \in [k] \times [k],
\end{align}
so that likelihood ratios are well-defined.
For each $\ell \in [m]$ and $u \in [k] \times [k]$, define the entrywise likelihood ratio
\begin{align}
    \s{L}_{\ell, u}(x) 
    \triangleq 
    \frac{\calP_{\ell, u}(x)}{\calQ(x)},
    \qquad
    \textnormal{so that } \qquad
    \E_{\calH_0}\pp{\s{L}_{\ell, u}(\s{X}_{ij})} = 1.
\end{align}

Then the conditional likelihood ratio (the Radon--Nikodym derivative of $\P_{\s{K}, \beta}$ with respect to $\P_{\calH_0}$) is
\begin{align}
    \s{L}(\s{X} \vert \s{K}, \beta) 
    = 
    \prod_{\s{B} \in \s{K}} 
    \prod_{(i,j) \in \s{B}} 
    \s{L}_{\beta(\s{B}), \varphi_{\s{B}}(i,j)}(\s{X}_{ij}),
    \label{eq:conditional_LR}
\end{align}
and the mixture likelihood ratio is
\begin{align}
    \s{L}_n(\s{X}) 
    \triangleq 
    \frac{\P_{\calH_1}}{\P_{\calH_0}}(\s{X})
    = 
    \E_{\s{K}, \beta} \pp{\s{L}(\s{X} \vert \s{K}, \beta)}.
\end{align}

\paragraph{General second moment reduction.}
We are now in a position to compute the second moment. Let $(\s{K}', \beta')$ be an independent copy of $(\s{K}, \beta)$. Then, by Fubini's theorem,

\begin{align}
    \E_{\calH_0}\pp{ \s{L}_n^2(\s{X})} 
    &= \E_{(\s{K}, \beta) \indep (\s{K}' , \beta')} \pp{\s{L}\p{\s{X} \vert \s{K}, \beta} \s{L}\p{\s{X} \vert \s{K}', \beta'}}.
\end{align}
Since under $\P_{\calH_0}$ the entries $\{\s{X}_{ij}\}_{i,j \in [n]}$ are independent, the inner expectation factorizes entry-wise. All coordinates $(i,j)$ outside the overlap of the planted unions contribute factor $1$, since $\E_{\calH_0}[\s{L}_{\ell,u}(\s{X}_{ij})]=1$ for all $(\ell,u)$. Only indices $(i,j)$ that lie in both planted unions contribute non-trivially. Define the overlap set and its size by
\begin{align}
    \s{K} \cap \s{K}' \triangleq \p{\bigcup_{\s{B} \in \s{K}}\s{B}} \bigcap \p{\bigcup_{\s{B}' \in \s{K}'}\s{B}'} = \bigcup_{(\s{B},\s{B}')\in \s{K} \times \s{K}'} \s{B} \cap \s{B}', \qquad \s{H} = \abs{\s{K} \cap \s{K}'},
\end{align}
Since blocks are disjoint within each planted collection, for each $(i,j)\in \s{K}\cap\s{K}'$ there exist unique blocks $\s{B}\in\s{K}$ and $\s{B}'\in\s{K}'$ such that $(i,j)\in \s{B}\cap\s{B}'$.
Therefore,

\begin{align}
    \E_{\calH_0}\pp{ \s{L}_n^2(\s{X})} = \E_{(\s{K}, \beta) \indep (\s{K}', \beta')} \pp{\prod_{(\s{B},\s{B}')\in \s{K} \times \s{K}'} \prod_{(i,j)\in\s{B}\cap\s{B}'} \rho\p{\beta(\s{B}), \beta'(\s{B}'); \varphi_{\s{B}}\p{i,j} \varphi_{\s{B}'}\p{i,j}}},
\end{align}
where
\begin{align}
    \rho(\ell,\ell'; u, u')
    \triangleq 
    \E_{\s{Z}\sim\calQ}\pp{\s{L}_{\ell,u}(\s{Z})\,\s{L}_{\ell',u'}(\s{Z})}.
\end{align}

\begin{lemma}[Cauchy--Schwarz domination of the overlap factor]
\label{lem:Cauchy_Schwarz_overlap}
Assume $\calP_{\ell,u} \ll \calQ$ for all $\ell \in [m]$, $u \in [k] \times [k]$, and write $\chi^2_{\ell,u}\triangleq \chi^2(\calP_{\ell,u}\Vert\calQ)$. 
Then, for all $(\ell,u),(\ell',u')$,
\begin{align}
    \rho(\ell,\ell'; u, u') 
    \le 
    \exp\ppp{\frac{1}{2}\p{\chi^2_{\ell, u} + \chi^2_{\ell', u'}}}.
\end{align}
\end{lemma}
\begin{proof}
By Cauchy--Schwarz,
\begin{align}
    \rho(\ell,\ell';u,u')
    &=
    \int \s{L}_{\ell,u}(x)\s{L}_{\ell',u'}(x)\,\calQ(x)\mathrm{d}x \\
    &\le
    \p{\int \s{L}_{\ell,u}^2(x)\calQ(x)\mathrm{d}x}^{1/2}
    \p{\int \s{L}_{\ell',u'}^2(x)\calQ(x)\mathrm{d}x}^{1/2} \\
    &=
    \sqrt{1+\chi^2_{\ell,u}}\sqrt{1+\chi^2_{\ell',u'}} \\
    &\le
    \exp\ppp{\frac{1}{2}\chi^2_{\ell,u}+\frac{1}{2}\chi^2_{\ell',u'}},
\end{align}
using $\log(1+x)\le x$ for $x\ge0$. 
\end{proof}

Applying Lemma~\ref{lem:Cauchy_Schwarz_overlap} and then Cauchy--Schwarz to separate the contributions of $\beta$ and $\beta'$ yields
\begin{align}
    \E_{\calH_0}\pp{\s{L}^2_n(\s{X})} 
    &\le 
    \E_{(\s{K}, \beta) \indep \s{K}'}\pp{\exp\ppp{ \sum_{(\s{B},\s{B}') \in \s{K}\times\s{K}'}  \sum_{(i,j)\in \s{B} \cap \s{B}'} {\chi^2_{\beta(\s{B}), \varphi_{\s{B}}(i,j)}} }} 
\end{align}

For each block-pair $(\s{B},\s{B}')$, define
\begin{align}
    \s{S}_{\s{B},\s{B}'}(\beta)
    \triangleq
    \sum_{(i,j)\in\s{B}\cap\s{B}'}
    \chi^2_{\beta(\s{B}), \varphi_{\s{B}}(i,j)}.
\end{align}

Let
\begin{align}
    \calF \triangleq \sigma\p{\s{K}, \ppp{\s{H}_{\s{B},\s{B}'}: \s{B}\in\s{K}, \s{B}' \in \s{K}'}}, \qquad \s{H}_{\s{B},\s{B}'} \triangleq \abs{\s{B} \cap \s{B}'}.\label{eq:sigma_field}
\end{align}

Conditioning on $\calF$ and applying conditional H\"older (e.g. \cite[Ch.~4]{durrett2019probability}) with exponent $q=m^2$ gives
\begin{align}
    \E_{\beta \indep \s{K}'\vert \calF}\pp{\left. \exp\ppp{\sum_{(\s{B}, \s{B}') \in \s{K} \times \s{K}'} \s{S}_{\s{B},\s{B}'}(\beta)}     \right|\calF} 
    \le \prod_{(\s{B}, \s{B}')\in\s{K}\times \s{K}'}
    \p{ \E_{\beta \indep \s{K}' \vert \calF }\pp{\left.e^{q\,\s{S}_{\s{B},\s{B}'}(\beta)}\right|\calF}
    }^{1/q}. \label{eq:cond_holder_step}
\end{align}
Since $\beta$ is a uniform bijection $\s{K}\to[m]$, the random label $\beta(\s{B})$ of each fixed $\s{B}\in\s{K}$ is uniform on $[m]$, hence for each block-pair $(\s{B},\s{B}')$,
\begin{align}
    \E_{\beta \indep \s{K}' \vert \calF}\pp{\left.e^{q\,\s{S}_{\s{B},\s{B}'}(\beta)}\right|\calF}
    &=
    \frac{1}{m}\sum_{\ell=1}^m
    \E_{\s{K}' \vert \calF}\pp{\left.        \exp\ppp{ q \sum_{(i,j)\in\s{B}\cap\s{B}'}     \chi^2_{\ell,\varphi_{\s{B}}(i,j)} }  \right|\calF}. \label{eq:beta_average}
\end{align}

For each block-pair $(\s{B},\s{B}')$, we define the random index set
\begin{align}
    \s{U}_{\s{B},\s{B}'} \triangleq \ppp{\varphi_{\s{B}}(i,j): (i,j)\in\s{B}\cap\s{B}'} \subset [k] \times [k], \quad \text{such that }  \abs{\s{U}_{\s{B},\s{B}'}} = \s{H}_{\s{B}, \s{B}'}. 
    \label{eq:U_BB_def}
\end{align}
 

\begin{lemma}[Conditional exchangeability of overlap coordinates] \label{lem:conditional_uniform_inblock}
Fix a block pair $(\s{B},\s{B}') \in \s{K}\times\s{K}'$ under either the non-consecutive placement model or the circular consecutive placement model, and let $\calF$ be defined as in~\eqref{eq:sigma_field}. Then, conditional on $\calF$, for all $u,u'\in[k] \times [k]$,
\begin{align}
    \P\ppp{u \in \s{U}_{\s{B},\s{B}'} \vert \calF}
    =
    \P\ppp{u' \in \s{U}_{\s{B},\s{B}'} \vert \calF}.
    \label{eq:cond_inclusion_equal}
\end{align}
In particular, for all $u\in[k] \times [k]$,
\begin{align}
    \E\pp{\Ind\ppp{u \in \s{U}_{\s{B},\s{B}'}} \vert \calF}
    =
    \frac{\s{H}_{\s{B},\s{B}'}}{k^2}.
    \label{eq:cond_inclusion_value}
\end{align}
Consequently, for any deterministic function $f:[k] \times [k]\to\R$,
\begin{align}
    \E\pp{\left.
        \sum_{u\in\s{U}_{\s{B},\s{B}'}} f(u)
    \right|\calF}
    =
    \frac{\s{H}_{\s{B},\s{B}'}}{k^2}
    \sum_{u\in[k] \times [k]} f(u).
    \label{eq:cond_avg_f}
\end{align}
\end{lemma}

\begin{proof}[Proof of Lemma \ref{lem:conditional_uniform_inblock}]
Fix $(\s{B},\s{B}')$ and condition on $\calF$, so that $\s{K}$ and all overlap sizes $\s{H}_{\s{B},\s{B}'}$ are fixed. We first prove~\eqref{eq:cond_inclusion_equal}.

\emph{Non-consecutive model.}
Conditional on $\calF$, the remaining randomness is the placement of $\s{K}'$ subject to the overlap sizes.
Under the non-consecutive placement rule, the conditional law does not privilege any specific location inside a fixed block $\s{B}$: for any $u,u'\in[k] \times [k]$, there exists a relabeling of the rows and columns inside $\s{B}$ that maps the event $\{u\in\s{U}_{\s{B},\s{B}'}\}$ to $\{u'\in\s{U}_{\s{B},\s{B}'}\}$ while leaving $\calF$ unchanged. Hence the conditional inclusion probabilities are equal.

\emph{Circular consecutive model.}
In the circular consecutive placement model, the planted row-interval and column-interval of $\s{B}'$ are generated by uniform starting points on the corresponding cycles. Conditional on $\calF$, the resulting law of the overlap location inside $\s{B}$ is invariant under simultaneous cyclic shifts of the $k$ template rows and of the $k$ template columns. Since such shifts act transitively on $[k] \times [k]$, the conditional inclusion probabilities are equal for all $u,u'\in[k] \times [k]$.

This proves~\eqref{eq:cond_inclusion_equal}. Now,
\begin{align}
    \sum_{u\in[k] \times [k]}\Ind\ppp{u\in\s{U}_{\s{B},\s{B}'}}
    = \abs{\s{U}_{\s{B},\s{B}'}}
    = \abs{\s{B}\cap\s{B}'}
    = \s{H}_{\s{B},\s{B}'}.
\end{align}
Taking conditional expectations and using~\eqref{eq:cond_inclusion_equal} yields
\begin{align}
    k^2\,\E\pp{\Ind\ppp{u\in\s{U}_{\s{B},\s{B}'}} \vert \calF}
    =
    \s{H}_{\s{B},\s{B}'},
\end{align}
which implies~\eqref{eq:cond_inclusion_value}. Finally, by linearity,
\begin{align}
    \E\pp{\left.
        \sum_{u\in\s{U}_{\s{B},\s{B}'}} f(u)
    \right|\calF}
    &=
    \sum_{u\in[k] \times [k]} f(u)\,
    \E\pp{\Ind\ppp{u\in\s{U}_{\s{B},\s{B}'}} \vert \calF}
    =
    \frac{\s{H}_{\s{B},\s{B}'}}{k^2}
    \sum_{u\in[k] \times [k]} f(u),
\end{align}
proving~\eqref{eq:cond_avg_f}.
\end{proof}

Fix a block pair $(\s{B},\s{B}')$ and a label $\ell\in[m]$. Recall that $q=m^2$ and
\begin{align}
    \s{S}_{\s{B},\s{B}'}(\ell)
    =
    \sum_{u\in \s{U}_{\s{B},\s{B}'}} \chi^2_{\ell,u},
    \qquad
    \s{H}_{\s{B},\s{B}'} = |\s{U}_{\s{B},\s{B}'}|,
\end{align}
If $\s{H}_{\s{B},\s{B}'}=0$ then $\s{S}_{\s{B},\s{B}'}(\ell)=0$ and both sides below equal $1$. Assume henceforth that $\s{H}_{\s{B},\s{B}'}\ge 1$. Next, we apply Jensen's inequality for the convex function $e^{qx}$:
\begin{align}
    \exp\ppp{q \s{S}_{\s{B}, \s{B}'}} 
    = 
    \exp\ppp{q \sum_{u \in \s{U}_{\s{B},\s{B}'}} \chi^2_{\ell, u}} 
    \le 
    \frac{1}{\s{H}_{\s{B},\s{B}'}}\sum_{u \in \s{U}_{\s{B},\s{B}'}} 
    \exp\ppp{q \s{H}_{\s{B},\s{B}'}\chi^2_{\ell,u}}, \label{eq:jensen_blockpair}
\end{align}
where we use the convention $\frac{1}{|\calU|}\sum_{u\in\calU}g(u)=1$ when $\calU=\emptyset$. 

Taking conditional expectation with respect to $\s{K}'$ given $\calF$, and using Lemma~\ref{lem:conditional_uniform_inblock} with $f(u)=\exp(q\s{H}_{\s{B},\s{B}'}\chi^2_{\ell,u})$, yields
\begin{align}
    \E_{\s{K}' \vert \calF}\pp{\exp\ppp{q \s{S}_{\s{B},\s{B}'}(\ell)}}
    \le
    \frac{1}{k^2}\sum_{u\in[k] \times [k]}\exp\ppp{q\s{H}_{\s{B},\s{B}'}\chi^2_{\ell,u}}
    =
    A_\ell\p{\s{H}_{\s{B},\s{B}'}},
    \label{eq:A_ell_def}
\end{align}
where the deterministic function $A_{\ell}(h)$ is defined for $h \in [0,k^2]$ as follows
\begin{align}
    A_{\ell}(h) \triangleq \frac{1}{k^2} \sum_{u \in [k] \times [k]} e^{q h \chi^2_{\ell, u}}.
\end{align}

\begin{lemma}[Log-convex interpolation] \label{lem:Log_inequality}
    For each $\ell \in [m]$, the function $h\mapsto \log A_\ell(h)$ is convex on $[0,k^2]$ and satisfies $A_\ell(0)=1$. In particular, for every $h\in[0,k^2]$
    \begin{align}
        \log A_{\ell}(h) \le \frac{h}{k^2} \log A_{\ell}(k^2), \quad \text{equivalently} \quad A_{\ell}(h)\le \exp\ppp{\theta_{\ell}h}, \label{eq:Log_inequality}
    \end{align}
    where
    \begin{align}
        \theta_{\ell} \triangleq \frac{1}{k^2} \log A_{\ell}(k^2) = \frac{1}{k^2} \log \p{\frac{1}{k^2} \sum_{u \in [k] \times [k]} \exp\ppp{q k^2 \chi^2_{\ell,u}}}. \label{eq:theta_ell} 
    \end{align}
\end{lemma}
\begin{proof}[Proof of Lemma \ref{lem:Log_inequality}]
    Let $x_{\ell,u} \triangleq \exp\ppp{q \chi^2_{\ell,u}} > 0$. Then $A_{\ell} = \frac{1}{k^2} \sum_{u \in [k] \times [k]} x_{\ell,u}^h$, and $h\mapsto \log\sum_{u} x_{\ell,u}^{\,h}$ is convex as a log-sum-exp of affine functions $h\log x_{\ell,u}$. Hence $h\mapsto \log A_\ell(h)$ is convex as well. Since $A_\ell(0)=1$, convexity implies
    \begin{align}
        \log A_{\ell}(h) \le \frac{h}{k^2}\log A_{\ell}(k^2) + \p{1 - \frac{h}{k^2}}\log A_{\ell}(0) = \frac{h}{k^2}\log A_{\ell}(k^2).
    \end{align}
\end{proof}

Combining~\eqref{eq:A_ell_def} with Lemma~\ref{lem:Log_inequality} gives, for every $h\in[0,k^2]$,
\begin{align}
    \E_{\s{K}' \vert \calF}\pp{\exp\ppp{q \s{S}_{\s{B},\s{B}'}(\ell)}\vert \calF}
    \le A_\ell\p{\s{H}_{\s{B},\s{B}'}}
    \le \exp\ppp{\theta_\ell\,\s{H}_{\s{B},\s{B}'}}. \label{eq:blockpair_exp_bound}
\end{align}
Recall the definition of the effective $\chi^2$ energy in~\eqref{eq:Theta_star_def} 
\begin{align}
    \Theta^\star 
    \triangleq \max_{\ell\in[m]} \frac{\theta_{\ell}}{q} 
    = \max_{\ell \in [m]} \frac{1}{qk^2} \log\p{\frac{1}{k^2} \sum_{u \in [k] \times [k]} \exp\ppp{q k^2 \chi^2_{\ell, u}}}, \qquad q=m^2. \label{eq:theta_star}
\end{align}
Substituting~\eqref{eq:blockpair_exp_bound} into~\eqref{eq:beta_average} yields 
\begin{align}
    \E_{\beta \indep \s{K}' \vert \calF}\pp{\left.
        e^{q\s{S}_{\s{B},\s{B}'}(\beta)}
    \right|\calF}
    &=
    \frac{1}{m}\sum_{\ell=1}^m
    \E_{\s{K}'\vert\calF}\pp{\exp\ppp{q\s{S}_{\s{B},\s{B}'}(\ell)}\vert\calF}
    \nonumber\\
    &\le
    \frac{1}{m}\sum_{\ell=1}^m
    \exp\ppp{\theta_\ell\,\s{H}_{\s{B},\s{B}'}}
    \le
    \exp\ppp{q\Theta^\star\,\s{H}_{\s{B},\s{B}'}}.
    \label{eq:blockpair_theta_star}
\end{align}

Plugging this bound into~\eqref{eq:cond_holder_step} and then taking expectation over $\calF$ gives
\begin{align}
    \E_{\calH_0}\pp{\s{L}_n^2(\s{X})}
    \le \E_{\calF}\pp{\exp\ppp{ \Theta^\star \sum_{\s{B}\in\s{K}} \sum_{\s{B}'\in\s{K}'} \abs{\s{B}\cap\s{B}'}}}.\label{eq:LR_bound_with_Theta_star}
\end{align}
By construction, the exponential term in~\eqref{eq:LR_bound_with_Theta_star} is $\calF$-measurable.

\subsubsection{Non-consecutive placements}
Recall that under the non-consecutive placements model, each block is constructed as $\s{B} = \s{S} \times \s{T}$. For each block-pair $(\s{B}, \s{B}') \in \s{K} \times \s{K}'$, each constructed as $\s{B} = \s{S}\times \s{T}$, $\s{B}' = \s{S}' \times \s{T}'$, we define their row and column overlap sizes 
\begin{align}
    \s{R}_{\s{B},\s{B}'} \triangleq \abs{\s{S} \cap \s{S}'}, \quad 
    \s{C}_{\s{B},\s{B}'} \triangleq \abs{\s{T} \cap \s{T}'}, \qquad \text{such that } \abs{\s{B}\cap\s{B}'} = \s{R}_{\s{B},\s{B}'} \s{C}_{\s{B},\s{B}'}.
\end{align}
For a fixed pair $(\s{B},\s{B}')$ with $\s{B}=\s{S}\times\s{T}$ and $\s{B}'=\s{S}'\times\s{T}'$, the overlap sizes satisfy
\begin{align}
    \s{R}_{\s{B},\s{B}'} \sim \s{Hypergeometric}(n,k,k),
    \qquad
    \s{C}_{\s{B},\s{B}'} \sim \s{Hypergeometric}(n,k,k).
\end{align}
Since the row and column sets are sampled without replacement, the corresponding coordinates are negatively associated \cite{joag1983negative}. As $x\mapsto e^{\Theta^\star x}$ is increasing, it follows that
\begin{align}
    \E_{\calF}\pp{\exp\ppp{\Theta^{\star} \sum_{(\s{B},\s{B}') \in \s{K}\times \s{K}'} \abs{\s{B}\cap\s{B}'} }} &\le \prod_{(\s{B}, \s{B}')\in\s{K} \times \s{K}'} \E\pp{\exp\ppp{\Theta^\star \s{R}_{\s{B},\s{B}'} \s{C}_{\s{B},\s{B}'}}} \\
    &= \p{\E\pp{\exp\ppp{\Theta^\star \s{R} \s{C}}}}^{m^2},
\end{align}
where $\s{R},\s{C} \sim \s{Hypergeometric}(n, k, k)$. We use that a $\s{Hypergeometric}(n,k,k)$ random variable is stochastically dominated by $\s{W} \sim \s{Binomial}(k,k/n)$; see, e.g., \cite{hoeffding1963probability}. Thus, letting $\s{W}'$ be an independent copy of $\s{W}$,
\begin{align}
    \E \pp{\exp\ppp{  \Theta^{\star} \s{R} \s{C}}} 
    = \E \pp{\exp\ppp{  \Theta^{\star} \s{W} \s{W}'}}
    = \E\pp{\p{1 + \frac{k}{n}\p{e^{ \Theta^{\star} \s{W}}  - 1}}^k}.
\end{align}
Assume $\Theta^{\star} \le \frac{1}{k}$. Then $0 \le \Theta^\star \s{W} \le 1$, and we use the bound $e^{x} - 1 \le x + x^2$ for $0 \le x \le 1$. This yields
\begin{align}
    \E \pp{\exp\ppp{  \Theta^{\star} \s{W} \s{W}'}} 
    &\le 
    \E\pp{\p{1 + \frac{k}{n}\p{\Theta^{\star} \s{W} + (\Theta^{\star})^2 \s{W}^2}}^k} \\
    &\le 
    \E\pp{\p{1 + 2 \frac{k}{n} \Theta^{\star} \s{W}}^k} \\
    &\le 
    \E\pp{e^{2\frac{k^2}{n} \Theta^{\star}\s{W}}} \\
    &= 
    \p{1 + \frac{k}{n}\p{e^{2 \frac{k^2}{n} \Theta^{\star}} - 1}}^{k}.
\end{align}
Therefore,
\begin{align}
    \E_{\calH_0}\pp{\s{L}_n^2(\s{X})} 
    \le 
    \p{1 + \frac{k}{n}\p{e^{2 \frac{k^2}{n} \Theta^{\star}} - 1}}^{k m^2}.
\end{align}
This is at most $1 + \delta$ provided 
\begin{align}
    \frac{k}{n}\p{e^{2 \frac{k^2}{n} \Theta^{\star}} - 1} \le (1 + \delta)^{\frac{1}{km^2}} - 1.
\end{align}
Since $(1 + \delta)^{\frac{1}{k m^2}} - 1 \ge \frac{\log(1+\delta)}{k m^2}$ this implied by
\begin{align}
    \Theta^{\star} \le \frac{n}{2k^2} \log \p{1 + \frac{
    n \log (1+\delta)}{ m^2k^2}}.
\end{align}
Combining this with the condition $\Theta^\star \le \frac{1}{k}$ yields, we obtain that the second-moment is at most $1+\delta$ if
\begin{align}
    \Theta^\star &\le \min\p{\frac{1}{k}, \frac{n}{2k^2} \log \p{1 + \frac{
    n \log (1+\delta)}{ m^2 k^2}}}
    \label{eq:universal_condition} \\ 
    &= \min\p{\frac{1}{k} , \frac{n^2 \log(1 + \delta)}{2 m^2 k^4}},
\end{align}
where the last equality holds when $\frac{n^2}{ m^2 k^2} = o(1)$. This establishes the claimed bound for the non-consecutive model.

\subsubsection{Circular consecutive placements}
The key distinction from the non-consecutive case lies in the distribution of the overlap size $|\s{B}\cap\s{B}'|$ under the circular consecutive placement family. Recall that $\s{B}=\s{S}\times\s{T}$ and $\s{B}'=\s{S}'\times\s{T}'$. For a fixed pair $(\s{B},\s{B}')$, we have
\begin{align}
    |\s{B}\cap\s{B}'|
    =
    |\s{S}\cap\s{S}'|\cdot|\s{T}\cap\s{T}'|.
\end{align}
Under the circular consecutive model, the row-interval $\s{S}'$ is generated by a uniform starting point on the cycle of length $n$ (and similarly for $\s{T}'$). Therefore, for $k\le n/2$,
\begin{align}
    \P\p{|\s{S}\cap\s{S}'| = z} = 
    \begin{cases}
        \frac{n - 2k + 1}{n}, & \text{for } z=0,\\[2pt]
        \frac{2}{n}, & \text{for } z=1, \ldots, k-1,\\[2pt]
        \frac{1}{n}, & \text{for } z=k,
    \end{cases}
    \label{eq:distribution_Z}
\end{align}
and the same distribution holds for $|\s{T}\cap\s{T}'|$.
Let $\s{Z}$ and $\s{Z}'$ be independent random variables with distribution \eqref{eq:distribution_Z}. Since $x\mapsto e^{\Theta^\star x}$ is increasing, an application of negative association yields 
\begin{align}
    \E_{\calF} \pp{\exp\ppp{\Theta^{\star} \sum_{(\s{B},\s{B}') \in \s{K}\times \s{K}'} \abs{\s{B}\cap\s{B}'} }} 
    &\le \prod_{(\s{B},\s{B}') \in\s{K}\times \s{K}'} \E\pp{\exp\ppp{\Theta^\star \abs{\s{B} \cap \s{B}'}}} \\
    &= \p{ \E\pp{\exp\ppp{\Theta^\star \s{Z}\s{Z}'}}}^{m^2}.
\end{align}
Next,
\begin{align}
    \E\pp{\exp\ppp{\Theta^\star \s{Z}\s{Z}'}} 
    &= \E_{\s{Z}'} \pp{\frac{n -2k +1 }{n} + \frac{2}{n} \sum_{z=1}^{k-1}e^{\Theta^\star z \s{Z}'} + \frac{e^{\Theta^\star k \s{Z}'}}{n}} \\
    &\le \E_{\s{Z}'}\pp{\frac{n-2k+1}{n} + \frac{2(k-1)}{n}e^{\Theta^\star k \s{Z}'} + \frac{e^{\Theta^\star k \s{Z}'}}{n}} \\
    &\le \E_{\s{Z}'}\pp{\frac{n-2k}{n} + \frac{2k}{n}e^{\Theta^\star k \s{Z}'}} \\
    &= \frac{n-2k}{n} + \frac{2k}{n}\,\E\pp{e^{\Theta^\star k \s{Z}'}} \\
    &= \frac{n-2k}{n} + \frac{2k}{n}\p{\frac{n -2k +1 }{n} + \frac{2}{n} \sum_{z=1}^{k-1}e^{\Theta^\star z k } + \frac{e^{\Theta^\star k^2}}{n}}\\
    &\le  \frac{n-2k}{n} + \frac{2k}{n}\p{\frac{n -2k}{n} + \frac{2k}{n}e^{\Theta^\star k^2}}\\
    &= 1 + \frac{4k^2}{n^2} \p{e^{\Theta^\star k^2} - 1}.
\end{align}
Therefore,
\begin{align}
    \p{ \E\pp{\exp\ppp{\Theta^\star \s{Z}\s{Z}'}}}^{m^2} \le \p{1 + \frac{4k^2}{n^2} \p{e^{\Theta^\star k^2} - 1}}^{m^2}.
\end{align}
This is at most $1 + \delta$ if
\begin{align}
    \frac{4k^2}{n^2} \p{e^{\Theta^\star k^2} - 1} \le (1+\delta)^{\frac{1}{m^2}} - 1.
\end{align}
Since $(1 + \delta)^{\frac{1}{m^2}} - 1 \ge \frac{\log(1+\delta)}{m^2}$, it suffices that
\begin{align}
    \Theta^\star &\le \frac{1}{k^2} \log\p{1 + \frac{n^2 \log (1+ \delta)}{4k^2 m^2}}.\label{eq:wrap_impossibility}
\end{align}

\subsection{Proof of Corollary~\ref{cor:standard_consecutive}} 
Let $\P_{\calH_1}^{\circ}$ denote the mixture alternative under the circular consecutive placement family, and let $\P_{\calH_1}^{\s{con}}$ denote the mixture alternative under the standard consecutive placement family. 
The two models differ only in the distribution of the planted support. Under the circular model, each row and column interval is generated by a uniform starting index in $[n]$, allowing wrap-around. Under the standard model, only the $n-k+1$ non-wrap starting positions are allowed.

For a one-dimensional circular interval with start $i\sim \s{Unif}([n])$, the interval wraps if and only if $i\in\{n-k+2,\ldots,n\}$. Hence,
\begin{align}
    \P(\text{wrap}) = \frac{k-1}{n}.
\end{align}
For a block $\s{B}=\s{S}\times\s{T}$ in the circular model, the row and column starts are independent and uniform on $[n]$. Therefore,
\begin{align}
    \P(\s{B}\text{ wraps}) \le \P(\s{S}\text{ wraps})+\P(\s{T}\text{ wraps})=\frac{2(k-1)}{n},
\end{align}
where $\s{B}$ is said to wrap if either its row interval or its column interval wraps around the boundary. 
Since the support contains $m$ blocks,
\begin{align}
    \P(\exists\, \text{wrapped block}) \le \frac{2m(k-1)}{n}.
\end{align}

We couple the two support distributions as follows. Sample $\s{K}^{\circ}$ from the circular placement family. If no block wraps, set $\s{K}^{\s{con}}=\s{K}^{\circ}$. Otherwise, sample $\s{K}^{\s{con}}$ independently from the standard placement family. Then
\begin{align}
    \P(\s{K}^{\s{con}}\neq \s{K}^{\circ}) \le \frac{2m(k-1)}{n}.
\end{align}

Conditional on a support $\s{K}$, both models draw a uniform bijection $\beta:\s{K}\to[m]$ and then generate $\s{X}$ with independent entries according to the same rule on $\bigcup_{\s{B}\in\s{K}}\s{B}$ and according to $\calQ$ elsewhere. On the event ${\s{K}^{\s{con}}=\s{K}^{\circ}}$ we use the same labeling $\beta$ and the same draw of $\s{X}$ in both models, and hence
\begin{align}
    d_{\s{TV}}\p{\P_{\calH_1}^{\circ},\P_{\calH_1}^{\s{con}}}\le \P(\s{K}^{\s{con}}\neq \s{K}^{\circ})
    \le \frac{2m(k-1)}{n}.
\end{align}
By Theorem~\ref{thrm:lower_bound}, if \eqref{eq:consecutive_impossibility} holds then
\begin{align}
    d_{\s{TV}}\p{\P_{\calH_0},\P_{\calH_1}^{\circ}}=o(1).
\end{align}
Using the triangle inequality,
\begin{align}
    d_{\s{TV}}\p{\P_{\calH_0},\P_{\calH_1}^{\s{con}}}
    \le 
    d_{\s{TV}}\p{\P_{\calH_0},\P_{\calH_1}^{\circ}}
    +
    d_{\s{TV}}\p{\P_{\calH_1}^{\circ},\P_{\calH_1}^{\s{con}}}
    \le 
    o(1)+\frac{2m(k-1)}{n}.
\end{align}
If $mk=o(n)$, the right-hand side is $o(1)$, completing the proof.

\subsection{Proof of Corollary~\ref{cor:smooth_signal}} 
We work under the problem formulation of Section~\ref{sec:prob_form} and assume the smooth-signal regime from Definition~\ref{def:smooth_signal}. For each $\ell\in[m]$ and $u\in[k]\times[k]$, let 
$\vartheta_{\ell,u}$ denote the local deviation from $\calQ$ 
(the mean in the mean-shift model and the variance shift in the variance-shift model). Recall the definition of $\Theta^\star$ in~\eqref{eq:Theta_star_def} and the energy definition in~\eqref{eq:energy_def}
\begin{align}
    \calE_\ell \triangleq \sum_{u\in[k] \times [k]} \vartheta_{\ell,u}^2,
    \qquad
    \calE \triangleq \max_{\ell\in[m]} \calE_\ell.
\end{align}
By Definition~\ref{def:smooth_signal}, there exist constants 
$\vartheta_0>0$ and $C_{\s{sp}}>0$, independent of $n$, 
such that for all sufficiently large $n$, for every 
$\ell \in [m]$ and $u \in [k] \times [k]$,
\begin{align}
    \abs{\vartheta_{\ell,u}} \le \vartheta_0, \qquad \max_{u\in[k] \times [k]} \vartheta_{\ell,u}^2\le C_{\s{sp}} \frac{\calE_\ell}{k^2}.
\end{align}
In the smooth-signal regime, Corollary~\ref{cor:smooth_signal}
states the impossibility conditions directly in terms of $\calE$,
whereas Theorem~\ref{thrm:lower_bound} is formulated in terms of
$\Theta^\star$. We therefore begin by relating $\Theta^\star$ to $\calE$.

\paragraph{Bound $\Theta^\star$ by the energy in the smooth-signal regime.}
As in the proof of Theorem~\ref{thrm:lower_bound}, we write $\chi^2_{\ell,u}\triangleq \chi^2(\calP_{\ell,u}\Vert\calQ)$. Fix a constant $C_\chi\ge 0$ and \emph{assume for the moment} that the following quadratic per-entry bound holds for all $\ell\in[m]$ and $u\in[k]\times[k]$.
\begin{align}\label{eq:chi_quad_assume}
    \chi^2_{\ell;u}\le C_\chi\,\vartheta_{\ell,u}^2.
\end{align}
Under this assumption, we show that in the smooth-signal regime,
\begin{align}\label{eq:Theta_energy_bound_proof}
    \Theta^\star \le C_\chi C_{\s{sp}}\frac{\calE}{k^2}.
\end{align}
Fix $\ell\in[m]$ and write $\chi^2_{\max}(\ell)\triangleq \max_{u\in[k]\times[k]}\chi^2_{\ell;u}$.
Since $\exp(\cdot)$ is increasing, for every $u\in[k] \times [k]$ we have
\begin{align}
    \exp\ppp{m^2 k^2 \chi^2_{\ell,u}} \le 
    \exp\ppp{m^2 k^2 \chi^2_{\max}(\ell)},
\end{align}
hence
\begin{align}
    \frac{1}{k^2} \sum_{u \in [k]\times [k]} \exp \ppp{m^2 k^2 \chi^2_{\ell,u}}\le 
    \exp \ppp{m^2 k^2 \chi^2_{\max}(\ell)}.
\end{align}
Taking $\log$ and dividing by $m^2k^2$ gives
\begin{align}
    \frac{1}{m^2 k^2} \log \p{\frac{1}{k^2} \sum_{u \in [k] \times [k]} \exp \ppp{m^2 k^2 \chi^2_{\ell,u}}} \le \chi^2_{\max}(\ell). \label{eq:theta_star_bound_lem}
\end{align}
Maximizing over $\ell$ yields the clean bound
\begin{align}
    \Theta^\star \le \max_{\ell\in[m]}\max_{u\in[k]\times[k]}\chi^2_{\ell,u}.
\end{align}
Now apply the assumed per-entry bound \eqref{eq:chi_quad_assume} 
\begin{align}
    \Theta^\star
    \le C_\chi\max_{\ell\in[m]}\max_{u\in[k]\times[k]}\vartheta_{\ell,u}^2.
\end{align}
Finally, use the non-spikiness condition in the smooth-signal regime~\ref{item:ii_non_spike}, i.e., $\max_{u\in[k]\times[k]}\vartheta_{\ell,u}^2 \le C_{\s{sp}}\frac{\calE_\ell}{k^2}$ for all $\ell \in [m]$, which gives 
\begin{align}
    \Theta^\star
    \le C_\chi C_{\s{sp}}\max_{\ell}\frac{\calE_\ell}{k^2}
    = C_\chi C_{\s{sp}}\frac{\calE}{k^2}.
\end{align}
This establishes \eqref{eq:Theta_energy_bound_proof} assuming \eqref{eq:chi_quad_assume}.

\paragraph{Conclude Corollary~\ref{cor:smooth_signal} from Theorem~\ref{thrm:lower_bound}.} We treat the non-consecutive and circular consecutive cases separately.
\begin{enumerate}
    \item \emph{Non-consecutive placements.} Assume $\calE=o\p{k\wedge \frac{n^2}{m^2k^2}}$. Then by \eqref{eq:Theta_energy_bound_proof},
    \begin{align}
        \Theta^\star \le C_\chi C_{\s{sp}}\frac{\calE}{k^2} 
        = o\p{\frac{1}{k}\ \wedge\ \frac{n^2}{m^2k^4}}.
    \end{align}
    In particular, since $C_{\chi}$ and $C_{\s{sp}}$ are fixed constants and  $\log(1+\delta)>0$ is fixed, the above $o(\cdot)$ bound implies that for all sufficiently large $n$,
    \begin{align}
        \Theta^\star \le \min\p{\frac{1}{k},\frac{n^2\log(1+\delta)}{2m^2k^4}},
    \end{align}
    so the non-consecutive impossibility condition in Theorem~\ref{thrm:lower_bound} holds,
    and therefore $d_{\s{TV}}(\P_{\calH_0},\P_{\calH_1})=o(1)$.

    \item \emph{Circular consecutive placements.} Assume $\calE=o\p{\log\p{1+\frac{n^2}{k^2m^2}}}$. Then \eqref{eq:Theta_energy_bound_proof} gives
    \begin{align}
        \Theta^\star
        \le C_\chi C_{\s{sp}}\frac{\calE}{k^2} =  o\p{\frac{1}{k^2}\log\p{1+\frac{n^2}{k^2m^2}}},
    \end{align}
    and hence, for all large $n$, 
    \begin{align}
        \Theta^\star \le \frac{1}{k^2}\log\p{1+\frac{n^2\log(1+\delta)}{4k^2m^2}}.
    \end{align}
    Thus the circular consecutive impossibility condition in Theorem~\ref{thrm:lower_bound} holds, and therefore $d_{\s{TV}}(\P_{\calH_0},\P_{\calH_1})=o(1)$.
\end{enumerate}

\paragraph{Verify the per-entry quadratic bound \eqref{eq:chi_quad_assume} in the two Gaussian models.}
It remains to justify the temporary assumption \eqref{eq:chi_quad_assume}. 

\begin{enumerate}
    \item \emph{Mean-shift.}
    Here $\calQ=\calN(0,1)$ and $\calP_{\ell;u}=\calN((M_\ell)_u,1)$, so
    \begin{align}
        \chi^2_{\ell,u}= \exp\ppp{\p{M_\ell}_u^2}-1.
    \end{align}
    Under the uniform boundedness condition~\ref{item:i_boundness} $|(M_\ell)_u|\le \vartheta_0$, convexity of $e^x$ on 
    $[0,\vartheta_0^2]$ gives 
    \begin{align}
        e^{(M_\ell)_u^2}-1 \le \frac{e^{\vartheta_0^2}-1}{\vartheta_0^2}\,(M_\ell)_u^2,
    \end{align}
    so \eqref{eq:chi_quad_assume} holds with $C_\chi=\frac{e^{\vartheta_0^2}-1}{\vartheta_0^2}$ and $\vartheta_{\ell,u}=(M_\ell)_u$.

    \item \emph{Variance-shift.} Here $\calQ=\calN(0,1)$ and $\calP_{\ell;u}=\calN\p{(0,1+(\Sigma_\ell)_u}$, with $0\le(\Sigma_\ell)_u\le \vartheta_0<1$. A direct calculation gives 
    \begin{align}
        \chi^2_{\ell;u}=\frac{1}{\sqrt{1-(\Sigma_\ell)_u^2}}-1.
    \end{align}
    Let $g(t)=(1-t)^{-1/2}-1$ on $[0,1)$. Then $g(0)=0$ and $g'(t)=\frac12(1-t)^{-3/2}$ is increasing, hence for $t\in[0,\vartheta_0^2]$,
    \begin{align}
        g(t) \le \p{\sup_{s\in[0,\vartheta_0^2]}g'(s)}t = \frac{1}{2(1-\vartheta_0^2)^{3/2}}t.
    \end{align}
    Taking $t=(\Sigma_\ell)_u^2$ yields
    \begin{align}
        \chi^2_{\ell;u}\le \frac{1}{2(1-\vartheta_0^2)^{3/2}}\,(\Sigma_\ell)_u^2,
    \end{align}
    so \eqref{eq:chi_quad_assume} holds with
    $C_\chi=\frac{1}{2(1-\vartheta_0^2)^{3/2}}$ and $\vartheta_{\ell,u}=(\Sigma_\ell)_u$.
\end{enumerate}

Combining these three steps completes the proof.

\section{Conclusion and Future Directions}

We analyzed detection in a finite-template inhomogeneous submatrix model for Gaussian matrices. The model permits multiple planted submatrices with coordinate-dependent signal structure, interpolating between the classical homogeneous mean-shift setting and fully heterogeneous alternatives. We established information-theoretic lower bounds via a $\chi^2$ second-moment argument and matching upper bounds based on global and scan statistics. In the smooth-signal regime, and under the associated sparsity conditions on $(m,k,n)$, the upper and lower bounds coincide up to logarithmic factors for both non-consecutive and consecutive placement models.

Our results highlight several structural features of submatrix detection. In the non-consecutive model, the number of planted submatrices affects the statistical boundary, and the information-theoretic threshold can lie strictly below the scan-based detection threshold in a nontrivial parameter regime. In the consecutive model, the scan procedure attains the information-theoretic threshold up to logarithmic factors. Under mild regularity conditions, heterogeneous templates exhibit the same detection scaling as the homogeneous case when expressed in terms of the total signal energy. The classical homogeneous model appears as a special case within our framework.

Several directions remain open. In the non-consecutive setting, as can be seen from Theorems~\ref{thrm:Detect_upper}--\ref{thrm:Detect_upper_variance}, both the computationally efficient global test and the computationally expensive scan test are needed to characterize the statistical limit. This suggests a statistical--computational gap: there is a parameter regime in which detection is information-theoretically possible (e.g., via the scan test), but no efficient algorithm is currently known. It would be interesting to provide evidence for such a gap using, for example, the framework of low-degree polynomials (see, e.g., \cite{Hopkins18,10.1007/978-3-030-97127-4_1}). On the modeling side, while the finite-template assumption isolates structured heterogeneity, extending the analysis to fully heterogeneous signals without template constraints would require new tools to control the overlap structure and second moments. It is also natural to consider non-Gaussian noise models or alternatives arising from more general exponential families, and to ask whether the effective energy characterization persists in those settings. More broadly, the finite-template model offers a controlled setting in which to study how structured inhomogeneity influences statistical and computational limits in high-dimensional detection problems. Similar phenomena are likely to arise in related structured matrix models. Finally, although this paper focuses on the detection problem, the recovery variant (i.e., identifying the planted submatrix exactly or partially) is also of interest.

\bibliographystyle{alpha}
\bibliography{bibfile}

\end{document}